\documentclass[leqno,11pt]{amsart}
\usepackage{amsmath,amscd,amsthm,amsxtra}
\usepackage{epsfig,graphics,color,colortbl}
\usepackage{amssymb,latexsym}
\usepackage{mathabx}
\usepackage{mathrsfs,eucal,upgreek}
\usepackage[poly,all]{xy}
\usepackage{hyperref}
\usepackage{tikz-cd}
\usepackage{arydshln}
\usepackage{ytableau}
\usepackage{adjustbox}
\usepackage{listings}
\usepackage{dirtytalk}
\usepackage[draft]{todonotes}
\usepackage{datetime}
\usepackage{tabularx} 
\usepackage{lipsum}
\usepackage{marginnote}
\usepackage[normalem]{ulem}  
\usepackage[nocompress]{cite}  
\usepackage{xparse} 
\usepackage{enumerate}

\renewcommand{\baselinestretch}{1.2}

\newtheorem{thm}{\bf Theorem}[section]

\newtheorem{prop}[thm]{\bf Proposition}
\newtheorem{cor}[thm]{\bf Corollary}
\newtheorem{lem}[thm]{\bf Lemma}
\newtheorem{rem}[thm]{\bf Remark}
\newtheorem{ex}[thm]{\bf Example}

\numberwithin{equation}{section}

\newcommand{\mr}{\mathring}
\newcommand{\ot}{\otimes}
\newcommand{\wtd}{\widetilde}
\newcommand{\wh}{\widehat}
\newcommand{\mc}{\mathcal}
\newcommand{\mf}{\mathfrak}
\newcommand{\ov}{\overline}
\newcommand{\scr}[1]{\text{\scriptsize{#1}}}
\newcommand{\ttt}{\texttt}

\newcommand{\blue}[1]{{\color{blue}#1}}
\newcommand{\red}[1]{{\color{red}#1}}
\newcommand{\green}[1]{{\color{green}#1}}
\newcommand{\lightgray}[1]{{\color{lightgray}#1}}
\newcommand{\Z}{\mathbb{Z}}
\newcommand{\R}{\mathbb{R}}
\newcommand{\C}{\mathbb{C}}
\newcommand{\X}{\mathcal{X}}
\newcommand{\Pa}{\mathscr{P}}
\newcommand{\Se}{\mathscr{S}}
\newcommand{\I}{\mathscr{I}}
\newcommand{\Io}{\mathscr{I}_\circ}
\newcommand{\m}{\mathsf{m}}
\newcommand{\bk}{\mathbf{k}}
\newcommand{\g}{\mathfrak{g}}
\newcommand{\h}{\mathfrak{h}}
\newcommand{\qi}{{q_i}}
\newcommand{\fc}{{\rm Rep}\left(U_q(\wh{\g})\right)}
\newcommand{\mcM}{\mathcal{M}}
\newcommand{\mcY}{\mathcal{Y}}
\newcommand{\A}{\mathscr{A}}

\newcommand{\frakc}{\mathfrak{c}}
\newcommand{\frakb}{\mathfrak{b}}
\newcommand{\frakd}{\mathfrak{d}}

\usepackage{todonotes}
\newcommand{\jang}[1]{\todo[size=\tiny,color=red!30]{#1 \ \hfill --- Jang}}
\newcommand{\Jang}[1]{\todo[size=\tiny,inline,color=red!30]{#1 \ \hfill --- Jang}}

\begin{document}

\title[Path description of $q$-characters in type $C$]
{Path description for $q$-characters of fundamental modules in type $C$}

\author{IL-SEUNG JANG}
\address{Department of Mathematics, Incheon National University, Incheon 22012, Korea}
\email{ilseungjang@inu.ac.kr}

\thanks{This work was supported by Incheon National University Research Grant in 2022 (No.~2022-0379).}

\begin{abstract}
In this paper, we investigate the behavior of monomials in the $q$-characters of the fundamental modules over a quantum affine algebra of untwisted type C.
As a result, we give simple closed formulae for the $q$-characters of the fundamental modules in terms of sequences of vertices in $\mathbb{R}^2$, so-called paths, with an admissible condition. This may be viewed as a type C analog of the path description of $q$-characters in types A and B due to Mukhin--Young.
\end{abstract}

\maketitle


\section{Introduction}
Let $\g$ be a finite-dimensional simple Lie algebra over $\mathbb{C}$.
We denote by $I$ an index set of the simple roots for $\g$.
Let $\wh{\g}$ be the associated affine Lie algebra of untwisted type \cite{Kac}. Let $U_q(\wh{\g})$ be the subquotient of the quantum affine algebra corresponding to $\wh{\g}$, that is, without the degree operator.
The notion of $q$-characters for finite-dimensional $U_q(\wh{\g})$-modules, which is defined in \cite{FR99}, is a refined analog of the ordinary character of finite-dimensional $\g$-modules.
More precisely, the {\it $q$-character homomorphism} (simply called $q$-character) is an injective ring homomorphism from the Grothendieck ring of the category of finite-dimensional $U_q(\wh{\g})$-modules into a ring of Laurent polynomials with infinitely many variables, denoted by $Y_{i,a}$ for $i \in I$ and $a \in \mathbb{C}^\times$.

The $q$-character is a useful tool for studying finite-dimensional $U_q(\wh{\g})$-modules \cite{FR99, FM01, H10, H10b}. Furthermore, it
has been widely used and studied in several viewpoints, such as crystal bases \cite{Na03, Kas03}, cluster algebras \cite{HL10}, and integrable systems \cite{FR99}. 
Despite its importance, it still seems to be a challenging problem to find explicit {\it closed} formulae for the $q$-characters of {\it arbitrary} finite-dimensional simple $U_q(\wh{\g})$-modules,
so there are several works which provide formulae {\it depending on $\g$ and what kind of $U_q(\wh{\g})$-modules}.

For example, for simply-laced types,
it was known in \cite{Na03, Na10} that arbitrary simple modules could be essentially computed by an analog of the Kazhdan--Lusztig algorithm, where closed formulae for the $q$-characters of standard modules are given in terms of corresponding classical crystals. 
In \cite{CM05, CM06}, Chari--Moura gave closed formulae for the $q$-characters of fundamental modules when $\g$ is of classical type by using the invariance of a braid group action on $\ell$-weights of fundamental modules. Note that the invariance is available only when $\g$ is of classical type.

Since the $q$-character is defined as in \eqref{eq:q-char morphism}, it is natural to associate a monomial with a sequence of vertices in $\mathbb{R}^2$.
In \cite{MY12}, when $\g$ is of type A and B, Mukhin and Young provided nice formulae of  $q$-characters for a certain family of finite-dimensional $U_q(\wh{\g})$-modules, so-called snake modules, in terms of non-overlapping paths that are sets of disjoint sequences in $\mathbb{R}^2$ with some admissible conditions reflecting features of monomials depending on $\g$.

The purpose of this paper is to investigate explicitly the behavior of monomials in the $q$-characters of fundamental modules in type C by adopting Mukhin--Young combinatorial approach, where the fundamental modules are the finite-dimensional $U_q(\wh{\g})$-modules corresponding to the variables $Y_{i,a}$'s under Chari--Pressley's classification of finite-dimensional (type 1) simple $U_q(\wh{\g})$-modules \cite{CP94,CP95}.
We should remark that the explicit formulas for the $q$-characters of fundamental modules (in type C) are already known by \cite{KS,CM06,NN06, NN07a}, but our formula as shown below seems to be more accessible combinatorially than other formulae.

As a result, we give simple closed formulae for the $q$-characters of fundamental modules in terms of {\it admissible} paths.
More precisely, let $L(Y_{i,k})$ be the fundamental $U_q(\wh{\g})$-module corresponding to $Y_{i,k} := Y_{i,q^k}$, where $\g$ is of type $C_n$ and $(i,k) \in I \times \Z_+$. 
Then we define a set of sequences $(j, \ell_j)_{0 \le j \le 2n}$ in $\mathbb{R}^2$ satisfying \eqref{def:a path}, denoted by $\Pa_{i,k}$, and then collect paths with an admissible condition as follows:
\begin{equation*}
	\ov{\Pa}_{i,k}  =
    \left\{\, 
	(j, \ell_j)_{0 \le j \le 2n} \in \Pa_{i,k}  \, | \, \ell_j \le \ell_{N-j} \, \text{\,for $j \in I$} 
	\,\right\}.
\end{equation*}
Then we have the following path description of $\chi_q(L(Y_{i,k}))$, which is the main result of this paper.

\begin{thm}[Theorem \ref{thm:main1}]
    For $i \in I$ and $k \in \Z_+$, we have
    \begin{equation} \label{eq:path model intro}
    \begin{split}
    	\chi_q(L(Y_{i, k}))
		=
		\sum_{p\, \in\, \ov{\Pa}_{i,k} } \m (p),
	\end{split}
    \end{equation}
    where $\m$ is defined in \eqref{eq:monomial map}.
\end{thm}

\noindent
As a corollary, we verify their thin property that every coefficient of monomials is $1$, which was already known in \cite{KS} (cf.~\cite{H05}).

We would like to emphasize that the underlying combinatorics of paths is essentially identical to the one of type A in \cite{MY12}, except for the admissible condition to reflect features of monomials for type C.
Therefore, it is natural to ask whether there is a notion of non-overlapping paths as in \cite{MY12} to introduce a type C analog to snake modules. 
However, there is an example in which a monomial corresponding to an overlapping path is contained in the $q$-character of a simple module under the path description (see Section \ref{subsec:remarks} for the example). 

Nevertheless, it would be interesting to find a suitable condition on (overlapping) admissible paths in this paper to extend \eqref{eq:path model intro} to higher levels (at least Kirillov--Reshetikhin modules) from the viewpoint of the recent work \cite{GDL} in which the authors introduced a notion of compatible condition on paths that are allowed to overlap to describe the $q$-characters of Hernandez--Leclerc modules for type A.
This may be discussed elsewhere.

This paper is organized as follows.
In Section \ref{sec:paths}, we set up necessary combinatorics for type C, such as paths, corners of paths, and moves on paths following some conventions of \cite{MY12}.
In Section \ref{sec:q-chars}, we briefly review some definitions and background for $q$-characters.
In Section \ref{sec:path description}, we give a path description for the $q$-characters of fundamental modules, which is the main result of this paper.
Finally, in Section \ref{sec:examples}, we provide the path descriptions of all fundamental modules in type $C_3$ to illustrate our main result and give some remarks for further discussion.
\vskip 3mm

\noindent{\bf Acknowledgement.} 
The author would like to thank Se-jin Oh for many stimulating and helpful discussions; to Sin-Myung Lee for his valuable comments. 
He also would like to thank sincerely the anonymous referee for helpful and detailed comments.
This work was supported by Incheon National University Research Grant in 2022 (No.~2022-0379).
\smallskip

\noindent
{\bf Convention}. 
We denote by $\Z_+$ the set of non-negative integers. For a finite set $S$, we understand the symbol $|S|$ as the number of elements of $S$.
If $\bk$ is a ring or field, then we denote by $\bk^\times$ the set of all non-zero elements of $\bk$. 
We use the notation $\delta$ in which $\delta(\mathsf{p}) = 1$ if a statement $\mathsf{p}$ is true, and $0$ otherwise.
In particular, put $\delta(i=j) := \delta_{ij}$ for $i, j \in \Z$.
Let $z$ be an indeterminate. Then we define the $z$-numbers, $z$-factorial and $z$-binomial by
{\allowdisplaybreaks
	\begin{gather*}
		[m]_z=\frac{z^m-z^{-m}}{z-z^{-1}}\,\, (m \in \Z_+),\,\,
		[m]_z!=[m]_z[m-1]_z\cdots [1]_z \,\, (m \geq 1),\,\, [0]_z!=1, \\
		\begin{bmatrix} m \\ k \end{bmatrix}_z = \frac{[m]_z[m-1]_z\cdots [m-k+1]_z}{[k]_z}\quad (0 \leq k \leq m).
	\end{gather*}
}

\section{Paths, Corners, and Moves} \label{sec:paths}

\subsection{Paths and corners}
For $n \ge 1$,
let $I = \{\, 1, 2, \dots, n \,\}$, $N = 2n$, $\I = \{ i \in \mathbb{Z} \, | \, 0 \le i \le N \}$ and $\Io = \I \setminus \{ 0, N \}$.
Set
\begin{equation} \label{eq:X}
	\X = \left\{\, (i, k) \in \I \times \Z \,\, | \,\, i - k \equiv 1 \,\, \text{mod} \,\, 2 \,\right\}.
\end{equation}
We denote by $\Se$ the set of all sequences in $\X$.
Let us take $(i, k) \in \X$.
We define a subset $\Pa_{i,k} $ of $\Se$ consisting of $p = ((r, y_r))_{r \in \I}$ satisfying the following condition:
\begin{equation} \label{def:a path}
\begin{split}
    & y_0 = i+k, \quad y_{N} = N-i+k, \\
	& y_{r+1}-y_r \in \{-1,\,1\} \quad 
    \text{for $0 \le r \le N-1$.}
\end{split}
\end{equation}
\vskip 1mm

\noindent
Set $\Pa  = \bigsqcup_{(i, k) \in \X} \Pa_{i,k} $.
We call a sequence in $\Pa $ by a {\it path}.
If $(j, \ell)$ is a point in $p \in \Pa $, then we often write $(j, \ell) \in p$ simply.
\smallskip

For $k \in \Io$, put
\begin{equation*} 
    \overline{k} = 
    \begin{cases}
        k & \text{if $k \le n$,} \\
        N-k & \text{if $k > n$.}
    \end{cases}
\end{equation*}
For a path $p = ((r, y_r))_{r \in \I} \in \Pa $, we define the sets $C_{p, \pm} $ of {\it upper} and {\it lower corners} of $p$, respectively, by
\begin{equation*}
    \begin{split}
        C_{p,+}  &= \left\{ \, (r,y_r) \in \X \, | \, r \in \Io,\,\, y_{r-1} = y_r + 1 = y_{r+1} \, \right\},\\
        C_{p,-}  &= \left\{ \, (r,y_r) \in \X \, | \, r \in \Io,\,\, y_{r-1} = y_r - 1 = y_{r+1} \, \right\}.
    \end{split}
\end{equation*}

\begin{rem} \label{rem:upper and lower corners}
{\em 
For any path $p \in \Pa$, if $p$ has two upper (resp.~lower) corners, then it is straightforward to check by \eqref{def:a path} that $p$ should have a lower (resp.~upper) corner between them. 
In particular, if a path has no lower (resp.~upper) corners, it has at most one upper (resp.~lower) corner.
}
\end{rem}

The following lemma is directly proved by definition (cf.~\cite[Lemma 5.5]{MY12}).

\begin{lem} \label{lem:characterization of paths}
    Any path in $\Pa$ has at least one corner and it is completely characterized by its upper and lower corners.
\end{lem}

\subsection{Moves on paths} \label{subsec:moves}
In this section, we will define two operators on $\Pa $ to obtain new paths in $\Pa $, which are called {\em raising} and {\em lowering moves}, respectively.

\subsubsection{Lowering moves}
For $(i,k) \in \X$, we say that a path $p \in \Pa_{i,k} $ can be {\it lowered} at $(j,\ell) \in p$ if $(j, \ell-1) \in C_{p,+} $.
If a path $p$ can be lowered at $(j,\ell) \in p$, then we define a {\it lowering move} on $p$ at $(j, \ell)$ in which we obtain a new path in $\Pa_{i,k} $ denoted by $p\A_{\ov{j},\ell}^{-1}$.
If a path $p \in \Pa_{i,k} $ can be lowered at $(j,\ell) \in p$, then
we define $p\A_{\ov{j},\ell}^{-1}$ by
\begin{equation*}
    \begin{split}
        & p\A_{\ov{j},\ell}^{-1} 
        := \left( \, \dots,\, (j-1, \ell),\, (j, \ell+1),\, (j, \ell),\, \dots\, \right),
    \end{split}
\end{equation*}
where the lowering move is depicted as shown below:
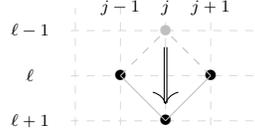
\begin{equation*}
    \begin{split}
        \begin{tikzpicture}[scale=0.6]
            \draw[help lines, color=gray!30, dashed] (-2.1,-1.1) grid (2.1, 1.5); 
            \node at (-1, 1.5) {\scalebox{0.8}{\scr{$j-1$}}};
            \node at (0, 1.5) {\scalebox{0.8}{\scr{$j$}}};
            \node at (1, 1.5) {\scalebox{0.8}{\scr{$j+1$}}};
            %
            \node at (-3, 1) {\scalebox{0.8}{\scr{$\ell-1$}}};
            \node at (-3, 0) {\scalebox{0.8}{\scr{$\ell$}}};
            \node at (-3, -1) {\scalebox{0.8}{\scr{$\ell+1$}}};
            \node (1) at (0, 1) {\lightgray{$\bullet$}};
            \node (2) at (0,-1) {$\bullet$};
            \node (3) at (-1,  0) {$\bullet$};
            \node (4) at (1,  0) {$\bullet$};
            %
            \draw[dashed, draw=lightgray] (3.center) -- (1.center);
            \draw[dashed, draw=lightgray] (4.center) -- (1.center);
            \draw[draw=lightgray] (3.center) -- (2.center);
            \draw[draw=lightgray] (4.center) -- (2.center);
            \draw[->, double] (1) -- (2);
        \end{tikzpicture}
    \end{split}
\end{equation*}
Note that it is straightforward to check that the new path obtained from a path $p \in \Pa_{i,k} $ by a lowering move is in $\Pa_{i,k} $.

\subsubsection{Raising moves}
For $(i,k) \in \X$, we say that a path $p \in \Pa_{i,k} $ can be {\it raised} at $(j,\ell) \in p$ if $(j, \ell+1) \in C_{p,-} $.
If a path $p$ can be raised at $(j,\ell) \in p$, then we define a {\it raising move} on $p$ at $(j, \ell)$ in an obvious reversed way of the lowering moves, in which we obtain a new path in $\Pa_{i,k} $ denoted by $p\A_{\ov{j},\ell}$.
We also notice that it is straightforward to check that the new path obtained from a path $p \in \Pa_{i,k} $ by a raising move is in $\Pa_{i,k} $.

\subsubsection{The highest and lowest paths of $\Pa_{i,k} $}
For $(i,k) \in \X$, we define the {\it highest} (resp.~{\it lowest}) path in $\Pa_{i,k} $ to be the path with no lower (resp.~upper) corners.

\begin{lem} \label{lem:highest and lowest paths}
    The highest and lowest paths in $\Pa_{i,k} $ are unique.
\end{lem}
\begin{proof}
	Let $(i,k) \in \X$. There is at least one standard choice of the highest and lowest paths in $\Pa_{i,k} $. To be more precise, we define a sequence $p = ((r, y_r))_{r \in \I}$ by $y_0 = i + k$, $y_N = N-i+k$, and 
	\begin{equation*}
		y_{r+1} - y_r = 
		\begin{cases}
			-1 & \text{if $0 \le r < i$,} \\
			\,\,\,\,\,1 & \text{if $i \le r \le N-1$.}
		\end{cases}
	\end{equation*}
	Clearly, $p \in \Pa_{i,k}$ and it has a unique upper corner. Moreover, $p$ has no lower corners by its definition. By Remark \ref{rem:upper and lower corners}, the uniqueness for the upper corner of $p$ follows. 
	The proof for lowest paths is similar (cf.~Example \ref{ex:highest and lowest corners}).
\end{proof}

\begin{ex} \label{ex:highest and lowest corners}
{\em 
Let $n = 3$.
Then the highest and lowest paths in $\Pa_{i,k}$ ($i \in I$ and $k \in \left\{ 0, 1 \right\}$ with $i-k \equiv 0 \mod{1}$) are drawn as follows, which are of the standard choice in the proof of Lemma \ref{lem:highest and lowest paths}:
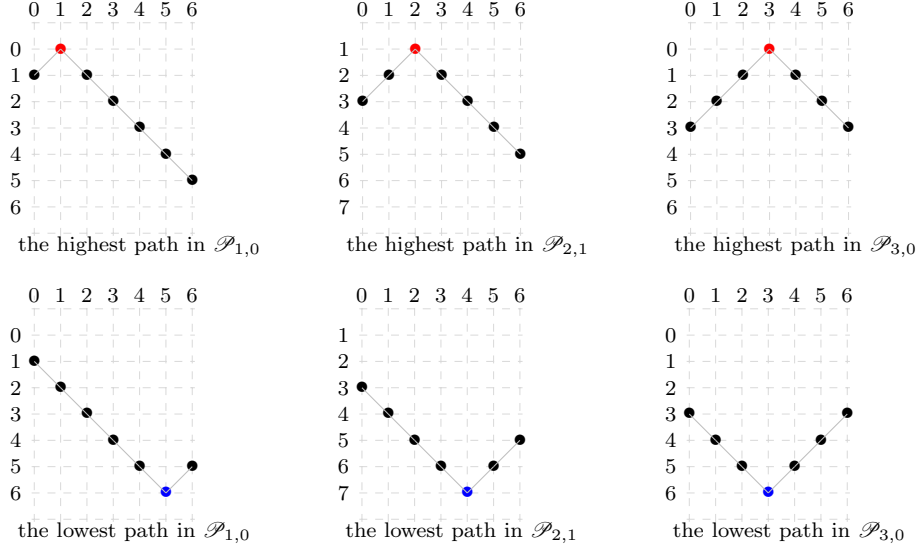
\begin{equation*}
\begin{split}
& \begin{tikzpicture}[scale=0.35, baseline=(current  bounding  box.center)]
            \draw[help lines, color=gray!30, dashed] (-3.1,-5.1) grid (3.1, 3.5); 
            %
            \node at (-3, 3.5) {\scr{$0$}};
            \node at (-2, 3.5) {\scr{$1$}};
            \node at (-1, 3.5) {\scr{$2$}};
            \node at (0, 3.5) {\scr{$3$}};
            \node at (1, 3.5) {\scr{$4$}};
            \node at (2, 3.5) {\scr{$5$}};
            \node at (3, 3.5) {\scr{$6$}};
            %
            \node at (-3.7, 2) {\scr{$0$}};
            \node at (-3.7, 1) {\scr{$1$}};
            \node at (-3.7, 0) {\scr{$2$}};
            \node at (-3.7, -1) {\scr{$3$}};
            \node at (-3.7, -2) {\scr{$4$}};
            \node at (-3.7, -3) {\scr{$5$}};
            \node at (-3.7, -4) {\scr{$6$}};
            %
            \node (0) at (-3,   1) {$\bullet$};
            \node (1) at (-2,   2) {\red{$\bullet$}};
            \node (2) at (-1,   1) {$\bullet$};
            \node (3) at ( 0,   0) {$\bullet$};
            \node (4) at ( 1,   -1) {$\bullet$};
            \node (5) at ( 2,   -2) {$\bullet$};
            \node (6) at ( 3,   -3) {$\bullet$};
            %
            \draw[draw=lightgray] (0.center) -- (1.center);
            \draw[draw=lightgray] (1.center) -- (2.center);
            \draw[draw=lightgray] (2.center) -- (3.center);
            \draw[draw=lightgray] (3.center) -- (4.center);
            \draw[draw=lightgray] (4.center) -- (5.center);
            \draw[draw=lightgray] (5.center) -- (6.center);
            \node at (1, -5.5) {\scr{\text{the highest path in $\Pa_{1,0}$}}};
	\end{tikzpicture}
	\qquad
	\begin{tikzpicture}[scale=0.35, baseline=(current  bounding  box.center)]
            \draw[help lines, color=gray!30, dashed] (-3.1,-5.1) grid (3.1, 3.5); 
            %
            \node at (-3, 3.5) {\scr{$0$}};
            \node at (-2, 3.5) {\scr{$1$}};
            \node at (-1, 3.5) {\scr{$2$}};
            \node at (0, 3.5) {\scr{$3$}};
            \node at (1, 3.5) {\scr{$4$}};
            \node at (2, 3.5) {\scr{$5$}};
            \node at (3, 3.5) {\scr{$6$}};
            %
            \node at (-3.7, 2) {\scr{$1$}};
            \node at (-3.7, 1) {\scr{$2$}};
            \node at (-3.7, 0) {\scr{$3$}};
            \node at (-3.7, -1) {\scr{$4$}};
            \node at (-3.7, -2) {\scr{$5$}};
            \node at (-3.7, -3) {\scr{$6$}};
            \node at (-3.7, -4) {\scr{$7$}};
            %
            \node (0) at (-3,   0) {$\bullet$};
            \node (1) at (-2,   1) {$\bullet$};
            \node (2) at (-1,   2) {\red{$\bullet$}};
            \node (3) at ( 0,   1) {$\bullet$};
            \node (4) at ( 1,   0) {$\bullet$};
            \node (5) at ( 2,   -1) {$\bullet$};
            \node (6) at ( 3,   -2) {$\bullet$};
            %
            \draw[draw=lightgray] (0.center) -- (1.center);
            \draw[draw=lightgray] (1.center) -- (2.center);
            \draw[draw=lightgray] (2.center) -- (3.center);
            \draw[draw=lightgray] (3.center) -- (4.center);
            \draw[draw=lightgray] (4.center) -- (5.center);
            \draw[draw=lightgray] (5.center) -- (6.center);
            \node at (1, -5.5) {\scr{\text{the highest path in $\Pa_{2,1}$}}};
	\end{tikzpicture}
	\qquad
	\begin{tikzpicture}[scale=0.35, baseline=(current  bounding  box.center)]
            \draw[help lines, color=gray!30, dashed] (-3.1,-5.1) grid (3.1, 3.5); 
            %
            \node at (-3, 3.5) {\scr{$0$}};
            \node at (-2, 3.5) {\scr{$1$}};
            \node at (-1, 3.5) {\scr{$2$}};
            \node at (0, 3.5) {\scr{$3$}};
            \node at (1, 3.5) {\scr{$4$}};
            \node at (2, 3.5) {\scr{$5$}};
            \node at (3, 3.5) {\scr{$6$}};
            %
            \node at (-3.7, 2) {\scr{$0$}};
            \node at (-3.7, 1) {\scr{$1$}};
            \node at (-3.7, 0) {\scr{$2$}};
            \node at (-3.7, -1) {\scr{$3$}};
            \node at (-3.7, -2) {\scr{$4$}};
            \node at (-3.7, -3) {\scr{$5$}};
            \node at (-3.7, -4) {\scr{$6$}};
            %
            \node (0) at (-3,   -1) {$\bullet$};
            \node (1) at (-2,   0) {$\bullet$};
            \node (2) at (-1,   1) {$\bullet$};
            \node (3) at ( 0,   2) {\red{$\bullet$}};
            \node (4) at ( 1,   1) {$\bullet$};
            \node (5) at ( 2,   0) {$\bullet$};
            \node (6) at ( 3,   -1) {$\bullet$};
            %
            \draw[draw=lightgray] (0.center) -- (1.center);
            \draw[draw=lightgray] (1.center) -- (2.center);
            \draw[draw=lightgray] (2.center) -- (3.center);
            \draw[draw=lightgray] (3.center) -- (4.center);
            \draw[draw=lightgray] (4.center) -- (5.center);
            \draw[draw=lightgray] (5.center) -- (6.center);
            \node at (1, -5.5) {\scr{\text{the highest path in $\Pa_{3,0}$}}};
	\end{tikzpicture}
	\\
	&		\begin{tikzpicture}[scale=0.35, baseline=(current  bounding  box.center)]
            \draw[help lines, color=gray!30, dashed] (-3.1,-5.1) grid (3.1, 3.5); 
            %
            \node at (-3, 3.5) {\scr{$0$}};
            \node at (-2, 3.5) {\scr{$1$}};
            \node at (-1, 3.5) {\scr{$2$}};
            \node at (0, 3.5) {\scr{$3$}};
            \node at (1, 3.5) {\scr{$4$}};
            \node at (2, 3.5) {\scr{$5$}};
            \node at (3, 3.5) {\scr{$6$}};
            %
            \node at (-3.7, 2) {\scr{$0$}};
            \node at (-3.7, 1) {\scr{$1$}};
            \node at (-3.7, 0) {\scr{$2$}};
            \node at (-3.7, -1) {\scr{$3$}};
            \node at (-3.7, -2) {\scr{$4$}};
            \node at (-3.7, -3) {\scr{$5$}};
            \node at (-3.7, -4) {\scr{$6$}};
            %
            \node (0) at (-3,   1) {$\bullet$};
            \node (1) at (-2,   0) {$\bullet$};
            \node (2) at (-1,   -1) {$\bullet$};
            \node (3) at ( 0,   -2) {$\bullet$};
            \node (4) at ( 1,   -3) {$\bullet$};
            \node (5) at ( 2,   -4) {\blue{$\bullet$}};
            \node (6) at ( 3,   -3) {$\bullet$};
            %
            \draw[draw=lightgray] (0.center) -- (1.center);
            \draw[draw=lightgray] (1.center) -- (2.center);
            \draw[draw=lightgray] (2.center) -- (3.center);
            \draw[draw=lightgray] (3.center) -- (4.center);
            \draw[draw=lightgray] (4.center) -- (5.center);
            \draw[draw=lightgray] (5.center) -- (6.center);
            \node at (0.8, -5.5) {\scr{\text{the lowest path in $\Pa_{1,0}$}}};
	\end{tikzpicture}
	\qquad\,\,
			\begin{tikzpicture}[scale=0.35, baseline=(current  bounding  box.center)]
            \draw[help lines, color=gray!30, dashed] (-3.1,-5.1) grid (3.1, 3.5); 
            %
            \node at (-3, 3.5) {\scr{$0$}};
            \node at (-2, 3.5) {\scr{$1$}};
            \node at (-1, 3.5) {\scr{$2$}};
            \node at (0, 3.5) {\scr{$3$}};
            \node at (1, 3.5) {\scr{$4$}};
            \node at (2, 3.5) {\scr{$5$}};
            \node at (3, 3.5) {\scr{$6$}};
            %
            \node at (-3.7, 2) {\scr{$1$}};
            \node at (-3.7, 1) {\scr{$2$}};
            \node at (-3.7, 0) {\scr{$3$}};
            \node at (-3.7, -1) {\scr{$4$}};
            \node at (-3.7, -2) {\scr{$5$}};
            \node at (-3.7, -3) {\scr{$6$}};
            \node at (-3.7, -4) {\scr{$7$}};
            %
            \node (0) at (-3,   0) {$\bullet$};
            \node (1) at (-2,   -1) {$\bullet$};
            \node (2) at (-1,   -2) {$\bullet$};
            \node (3) at ( 0,   -3) {$\bullet$};
            \node (4) at ( 1,   -4) {\blue{$\bullet$}};
            \node (5) at ( 2,   -3) {$\bullet$};
            \node (6) at ( 3,   -2) {$\bullet$};
            %
            \draw[draw=lightgray] (0.center) -- (1.center);
            \draw[draw=lightgray] (1.center) -- (2.center);
            \draw[draw=lightgray] (2.center) -- (3.center);
            \draw[draw=lightgray] (3.center) -- (4.center);
            \draw[draw=lightgray] (4.center) -- (5.center);
            \draw[draw=lightgray] (5.center) -- (6.center);
            \node at (0.8, -5.5) {\scr{\text{the lowest path in $\Pa_{2,1}$}}};
	\end{tikzpicture}
	\qquad\,\,
			\begin{tikzpicture}[scale=0.35, baseline=(current  bounding  box.center)]
            \draw[help lines, color=gray!30, dashed] (-3.1,-5.1) grid (3.1, 3.5); 
            %
            \node at (-3, 3.5) {\scr{$0$}};
            \node at (-2, 3.5) {\scr{$1$}};
            \node at (-1, 3.5) {\scr{$2$}};
            \node at (0, 3.5) {\scr{$3$}};
            \node at (1, 3.5) {\scr{$4$}};
            \node at (2, 3.5) {\scr{$5$}};
            \node at (3, 3.5) {\scr{$6$}};
            %
            \node at (-3.7, 2) {\scr{$0$}};
            \node at (-3.7, 1) {\scr{$1$}};
            \node at (-3.7, 0) {\scr{$2$}};
            \node at (-3.7, -1) {\scr{$3$}};
            \node at (-3.7, -2) {\scr{$4$}};
            \node at (-3.7, -3) {\scr{$5$}};
            \node at (-3.7, -4) {\scr{$6$}};
            %
            %
            \node (0) at (-3,   -1) {$\bullet$};
            \node (1) at (-2,   -2) {$\bullet$};
            \node (2) at (-1,   -3) {$\bullet$};
            \node (3) at ( 0,   -4) {\blue{$\bullet$}};
            \node (4) at ( 1,   -3) {$\bullet$};
            \node (5) at ( 2,   -2) {$\bullet$};
            \node (6) at ( 3,   -1) {$\bullet$};
            %
            \draw[draw=lightgray] (0.center) -- (1.center);
            \draw[draw=lightgray] (1.center) -- (2.center);
            \draw[draw=lightgray] (2.center) -- (3.center);
            \draw[draw=lightgray] (3.center) -- (4.center);
            \draw[draw=lightgray] (4.center) -- (5.center);
            \draw[draw=lightgray] (5.center) -- (6.center);
            \node at (0.8, -5.5) {\scr{\text{the lowest path in $\Pa_{3,0}$}}};
	\end{tikzpicture}
\end{split}
\end{equation*}
Here an upper (resp.~lower) corner is marked as a red (resp.~blue) bullet.
}
\end{ex}

\section{Quantum affine algebras and $q$-characters} \label{sec:q-chars}
\subsection{Cartan data}
Let $C = (a_{ij})_{i, j \in I}$ be a Cartan matrix of finite type, where $I = \{ \, 1, 2, \dots, n \, \}$.
Since $C$ is symmetrizable, there exists a diagonal matrix $D = (d_i)_{i \in I}$ such that $DC$ is non-zero symmetric and $d_i \in \Z_+$ for all $i \in I$. 
Let $q \in \C^\times$ be not a root of unity. Then we set $q_i = q^{d_i}$.
Let us take a realization $(\h, \Pi, \Pi^\vee)$ over $\C$ with respect to $C$, where $\Pi = \{ \alpha_1, \dots, \alpha_n \} \subset \h$ (set of simple roots) and $\Pi^\vee = \{ \, \alpha_1^\vee, \dots, \alpha_n^\vee \, \} \subset \h^*$ (set of simple coroots) so that $\alpha_j(\alpha_i^\vee) = a_{ij}$ for $i, j \in I$.
We denote by $P = \{ \, \lambda \in \h^* \, | \, \lambda(\alpha_i^\vee) \in \Z \,\, \text{for all $i \in I$} \, \}$ the set of weights and by $P^+ = \{ \, \lambda \in P \, | \, \lambda(\alpha_i^\vee)  \ge 0 \,\, \text{for all $i \in I$} \, \}$ the set of dominant weights.
Let $\varpi_1, \dots, \varpi_n \in \h^*$ (resp. $\varpi_1^\vee, \dots, \varpi_n^\vee \in \h$) be the fundamental weights (resp. coweights) defined by $\varpi_i(\alpha_j^\vee) = \alpha_j(\varpi_i^\vee) = \delta_{i,j}$ for $i, j \in I$.
Let $Q$ (resp. $Q^+$) be the $\Z$-span (resp. $\Z_+$-span) of $\Pi$.
We take the partial order $\le$ on $P$ with respect to $Q^+$ in which $\lambda \le \lambda'$ if and only if $\lambda' - \lambda \in Q^+$ for $\lambda, \lambda' \in P$.
We denote by $\g$ the  complex simple Lie algebra with respect to $C$.
Let $\wh{\g}$ be the {\it untwisted affine algebra} corresponding to $\g$ \cite{Kac}.

\subsection{Quantum affine algebras}
The {\em quantum affine algebra} $U_q(\wh{\g})$ is the associative algebra over $\C$ generated by $e_i$, $f_i$, and $k^{\pm 1}_i$ for $0 \le i \le n$, and $C^{\pm \frac{1}{2}}$ subject to the following relations:
{\allowdisplaybreaks
\begin{gather*}
	\text{$C^{\pm \frac{1}{2}}$ are central with $C^{\frac{1}{2}}C^{-\frac{1}{2}}=C^{-\frac{1}{2}}C^{\frac{1}{2}} = 1$,} \\
	k_ik_j=k_jk_i,\quad k_ik_i^{-1}=k_i^{-1}k_i=1,\quad \prod_{i=0}^n k_i^{\pm a_i} = ( C^{\pm\frac{1}{2}} )^2, \\
	k_ie_jk_i^{-1}=\qi^{a_{ij}}e_j,\quad k_if_jk_i^{-1}=\qi^{-a_{ij}}f_j,\\
	e_if_j-f_je_i=\delta_{ij}\frac{k_i-k_i^{-1}}{\qi-\qi^{-1}},\\
	\sum_{m=0}^{1-a_{ij}}(-1)^{m}e_i^{(1-a_{ij}-m)}e_je_i^{(m)}=0,\quad\quad
	\sum_{m=0}^{1-a_{ij}}(-1)^{m}f_i^{(1-a_{ij}-m)}f_jf_i^{(m)}=0\quad(i\ne j),
\end{gather*}
}
\noindent \!\!for $0 \le i,j \le n$, where $e_i^{(m)}=e_i^m/[m]_i!$ and $f_i^{(m)}=f_i^m/[m]_i!$ for $0 \le i \le n$ and $m \in \Z_+$.
There is a Hopf algebra structure on $U_q(\wh{\g})$, where the comultiplication $\Delta$ and the antipode $S$ are given by
\begin{gather*}
	\Delta(k_i)=k_i \ot k_i, \quad
	\Delta(e_i)= e_i \ot 1 + k_i\ot e_i, \quad
	\Delta(f_i)= f_i \ot k_i^{-1} + 1 \ot f_i,\\
	S(k_i)=k_i^{-1}, \ \ S(e_i)=-k_i^{-1} e_i, \ \  S(f_i)=-f_i k_i,
\end{gather*}
for $0 \le i \le n$.
It is well-known in \cite{B94} that as a Hopf algebra, $U_q(\wh{\g})$ is also isomorphic to the algebra generated by
$x_{i,r}^\pm$ ($i\in I, r\in \Z$), $k_{i}^{\pm 1}$ $(i\in I)$, $h_{i,r}$ ($i\in I, r\in \Z^\times$), and $C^{\pm \frac{1}{2}}$ subject to the following relations:
{\allowdisplaybreaks
\begin{gather*}
	\text{$C^{\pm \frac{1}{2}}$ are central with $C^{\frac{1}{2}}C^{-\frac{1}{2}}=C^{-\frac{1}{2}}C^{\frac{1}{2}} = 1$,} \\
	k_ik_j=k_jk_i,\quad k_ik_i^{-1}=k_i^{-1}k_i=1,\\
	k_ih_{j,r}=h_{j,r}k_i,\quad k_ix^{\pm}_{j,r}k_i^{-1}=\qi^{\pm a_{ij}}x^{\pm}_{j,r},\\
	[h_{i,r},h_{j,s}]=\delta_{r,-s}\frac{1}{r}[r a_{ij}]_i\frac{C^r-C^{-r}}{\qi-\qi^{-1}},\\
	[h_{i,r},x^{\pm}_{j,s}]=\pm \frac{1}{r}[r a_{ij}]_i C^{\mp |r|/2}x^{\pm}_{j,r+s},\\
	x^\pm_{i,r+1}x^\pm_{j,s}-\qi^{\pm a_{ij}}x^\pm_{j,s}x^\pm_{i,r+1}
	= \qi^{\pm a_{ij}}x^\pm_{i,r}x^\pm_{j,s+1}-x^\pm_{j,s+1}x^\pm_{i,r},\\
	[x^+_{i,r},x^-_{j,s}]=\delta_{i,j}\frac{C^{(r-s)/2}\psi^+_{i,r+s}-C^{-(r-s)/2}\psi^-_{i,r+s}}{\qi-\qi^{-1}},\\
	\sum_{w\in \mf{S}_m}\sum_{k=0}^m {\small \left[\begin{matrix} m \\ k \end{matrix}\right]}_i x^\pm_{i,r_{w(1)}}\dots x^\pm_{i,r_{w(k)}}x^\pm_{j,s}x^\pm_{i,r_{w(k+1)}}\dots x^\pm_{i,r_{w(m)}}=0\quad (i\neq j),
\end{gather*}
}
\noindent \!\!where $r_1, \dots, r_m$ is any sequence of integers with $m=1-a_{ij}$, $\mf{S}_m$ denotes the group of permutations on $m$ letters, and $\psi^\pm_{i,r}$ is the element determined by the following identity of formal power series in $z$;
\begin{equation}\label{eq:psi generators}
	\sum_{r=0}^\infty \psi^\pm_{i,\pm r} z^{\pm r} = k_i^{\pm 1}\exp\left( \pm(\qi-\qi^{-1}) \sum_{s=1}^\infty h_{i,\pm s} z^{\pm s} \right).
\end{equation}
Here $\psi_{i, r}^+ = 0$ for $r < 0$, and $\psi_{i, r}^- = 0$ for $r > 0$.

\subsection{Finite-dimensional modules and $q$-characters}

Let $\mathcal{C}$ be the category of finite-dimensional $U_q(\wh{\g})$-modules. We denote by $\fc$ the Grothendieck ring for the category $\mathcal{C}$.
A finite-dimensional $U_q(\g)$-module $W$ is said to be of type $1$ if it admits the direct sum of its weight spaces of the following form:
\begin{equation*}
	W_\lambda = \left\{ \, w \in W \, | \, k_i w = q^{(\lambda, \alpha_i)} w \, \right\} \quad (\lambda \in P).
\end{equation*}
A finite-dimensional $U_q(\wh{\g})$-module $V$ is {\it of type $1$} if $C^{\pm \frac{1}{2}}$ acts as identity on $V$ and $V$ is of type $1$ as a $U_q(\g)$-module. 
According to \cite{CP94}, every finite-dimensional simple $U_q(\wh{\g})$-module can be obtained from a finite-dimensional $U_q(\wh{\g})$-module of type $1$ by twisting the actions of $U_q(\wh{\g})$ under an automorphism of $U_q(\wh{\g})$.
In what follows, all $U_q(\wh{\g})$-modules in this paper will be assumed to be of type 1 without further comment.

For a $U_q(\wh{\g})$-module $V$ in $\mathcal{C}$, since $\left\{ k_i^\pm, \, \psi_{i,\pm r}^\pm \, | \, i \in I,\, r \ge 0 \right\}$ is a commuting family as endomorphisms of $V$, we have
\begin{equation*}
    V = \bigoplus_{\Phi = (\phi_{i, \pm r}^{\pm})_{i \in I, r \ge 0}} V_\Phi,
\end{equation*}
where $V_\Phi = \left\{ v \in V \, | \, \text{$\exists p\ge 0$ such that $\left( \psi_{i,\pm r}^\pm - \phi_{i,\pm r}^\pm \right)^p v = 0$ for $i \in I$ and $r \ge 0$} \right\}$.
We call $\Phi$ a {\it $\ell$-weight} of $V$.
We often identify $\Phi$ with the sequence $(\Phi_i(z))_{i \in I}$, where 
\begin{equation*}
    \Phi_i(z) = \sum_{r \ge 0} \phi_{i, \pm r}^{\pm} z^{\pm r} \in \C[[z^{\pm 1}]].
\end{equation*}
It is known in \cite{FR99} that $\Phi_i(z)$ is of the following form in $\C[[z^{\pm 1}]]$:
\begin{equation} \label{eq:l-weight formula}
    \Phi_i(z) = q_i^{{\rm deg}(Q_i) - {\rm deg}(R_i)} \frac{Q_i(zq_i^{-1})R_i(zq_i)}{Q_i(zq_i)R_i(zq_i^{-1})},
\end{equation}
where $Q_i(z) = \prod_{a \in \C^\times} (1-az)^{r_{i,a}}$ and $R_i(z) = \prod_{a \in \C^\times} (1-az)^{s_{i,a}}$ for some $r_{i, a}, s_{i,a} \in \Z_+$.

The {\it $q$-character $\chi_q$} is an injective morphism from $\fc$ into the ring of Laurent polynomials of $Y_{i,a}$ $(i \in I, a \in \C^\times)$ \cite{FR99}:
\begin{equation} \label{eq:q-char morphism}
\begin{split}
    \chi_q : 
    \xymatrixcolsep{3pc}\xymatrixrowsep{0pc}
    \xymatrix{
  \fc \ \ar@{->}[r] & \Z[Y_{i,a}^{\pm 1}]_{i \in I, a \in \C^\times} \\
 [V] \ar@{|->}[r] & \displaystyle \sum_{\Phi} \dim ( V_{\Phi} ) \, m_\Phi
},
\end{split}
\end{equation}
where $m_{\Phi} = \prod_{i \in I, a \in \C^\times} Y_{i,a}^{r_{i,a}-s_{i,a}}$ with respect to $\ell$-weight $\Phi$ of the form \eqref{eq:l-weight formula}. 
The $m_\Phi$ are called {\it monomials} or {\it $\ell$-weights} associated to $\ell$-weights $\Phi$.
We denote by $\mcM$ the set of all monomials in $\Z[Y_{i,a}^{\pm 1}]_{i \in I, a \in \C^\times}$.
Let $J$ be a subset of $I$. 
For $m = \prod_{i \in I, a \in \C^\times} Y_{i, a}^{u_{i,a}(m)} \in \mcM$, the monomial $m$ is called {\it $J$-dominant} if $u_{j, a}(m) \ge 0$ for all $j \in J$. When $J = I$, we call $m$ {\it dominant monomial}.
Put $\mcM^+$ as the subset of $\mcM$ consisting of dominant monomials.

A finite-dimensional $U_q(\wh{\g})$-module $V$ is {\it of highest $\ell$-weight $m = m_{\Phi}$} if there exists a non-zero vector $v \in V$ such that 
\begin{center}
	\!\!\!\!
    (1) $V = U_q(\wh{\g})v$, \,
    (2) $e_i v = 0$ for all $i \in I$, \,
    (3) $\psi_{i, \pm r}^{\pm} v = \phi_{i,\pm r}^{\pm}v$ for $i \in I$, $r \ge 0$.
\end{center}
Under Chari--Pressley's classification \cite{CP95} of finite-dimensional simple $U_q(\wh{\g})$-modules, for $m \in \mcM^+$, there exists a unique finite-dimensional $U_q(\wh{\g})$-module of highest $\ell$-weight $m$, which is denoted by $L(m)$.
In the case of $m = Y_{i,a}$ for $i \in I$ and $a \in \C^\times$, we say that $L(Y_{i,a})$ is a {\it fundamental module}.

\begin{rem} \label{rem:for fundamental modules}
{\em 
Since any simple module in $\mathcal{C}$ is a subquotient of a tensor product of fundamental modules \cite{CP94,CP95}, the fundamental modules can be viewed as building blocks of $\mathcal{C}$. 
On the other hand, it is known that $\mathcal{C}$ is not semisimple (which is observed easily even in the case of $U_q(\widehat{\mathfrak{sl}}_2)$), and the tensor product of simple $U_q(\widehat{\g})$-modules in $\mathcal{C}$ still remains simple generically (e.g.~see \cite{AK97,FM01,Kas02}).
These leads to introduce a notion of prime object in $\mathcal{C}$ \cite{CP97}, namely,
we say that an simple module $S$ of $\mathcal{C}$ is {\it prime} if there exists no nontrivial factorization $S = S_1 \otimes S_2$, where $S_1$ and $S_2$ are (simple) modules of $\mathcal{C}$. 
One of goals in representation theory of $U_q(\widehat{\g})$ is to classify all prime simple modules in $\mathcal{C}$ and understand their tensor product structure. Note that the fundamental modules are prime simple modules.
}
\end{rem}

For $a \in \mathbb{C}^\times$, there exists a Hopf algebra automorphism of $U_q(\widehat{\g})$, denoted by $\tau_a$, defined by
\begin{equation*}
	\tau_a(x_{i,r}^\pm) = a^n x_{i,r}^\pm, \quad
	\tau_a(\psi_{i,r}^\pm) = a^n \psi_{i,r}^\pm, \quad
	\tau_a\left(C^{1/2}\right) = C^{1/2}, \quad
	\tau_a(k_i) = k_i.
\end{equation*}
We denote by $V(a)$ the pull-back of a module $V$ in $\mathcal{C}$ under $\tau_a$. 
It is known by \cite{CP94,CP95} (and also Theorem \ref{thm:prop1 of qchar} and \ref{thm:Image of qchar is the intersection of kernels of Si}) that the spectral parameters of all (dominant) monomials appearing in $\chi_q(V)$ can be shifted by $\tau_a$. 
This is the reason why we only consider $\chi_q\left(L(Y_{i,q^k})\right)$ $(i \in I, k \in \Z_+)$ throughout this paper (see also Remark \ref{rem:skeleton category Cz}).

\begin{rem} \label{rem:skeleton category Cz}
{\em 
Let $\mathcal{C}_\Z$ be the full subcategory of $\mathcal{C}$ whose objects $V$ satisfy that every composition factor of $V$ satisfies certain integral condition (see \cite[Section 3.7]{HL10} for more details). Then it is known (e.g.~see \cite{CP94, Kas02, Chari02}) that every simple module in $\mathcal{C}$ can be written as a tensor product $S_1(a_1) \otimes \dots S_k(a_k)$ for some simple modules $S_1, \dots, S_k \in \mathcal{C}_\Z$ and some $a_1, \dots, a_k \in \mathbb{C}$ such that $a_i / a_j \notin q^{2\Z}$ for $1 \le i < j \le k$.
Clearly, $L(Y_{i,q^k})$ $(i \in I, k \in \Z)$ is an object of $\mathcal{C}_\Z$.
}
\end{rem}

\subsection{Properties of $q$-characters}
Let us briefly review properties of $q$-character homomorphism \eqref{eq:q-char morphism} following \cite{FR99, FM01}.
For $m \in \mcM^+$, we denote by $\mcM(m)$ the set of all monomials in $\chi_q(L(m))$.
For $i \in I$ and $a \in \C^\times$, let
\begin{equation} \label{eq:variable A}
    A_{i,a} = Y_{i, aq_i^{-1}} Y_{i, aq_i} \prod_{a_{ji} = -1} Y_{j, a}^{-1} \prod_{a_{ji}=-2} Y_{j, aq^{-1}}^{-1} Y_{j, aq}^{-1} \prod_{a_{ji}=-3} Y_{j, aq^{-2}} Y_{j, a} Y_{j, aq^2}.
\end{equation}

\begin{thm}\!\!\!{\em \cite{FM01}} \label{thm:prop1 of qchar}
For $m \in \mcM^+$, we have
\begin{equation*}
    \mcM(m) \subset m \cdot \Z[A_{i, a}^{-1}]_{i \in I, a \in \C^\times}.
\end{equation*}
\end{thm}

\noindent
Let $\mcY = \Z[Y_{i, a}]_{i \in I, a \in \C^\times}$. 
For $i \in I$, we denote by $\wtd{\mcY}_i$ the free $\mcY$-module generated by $S_{i,x}$ for $x \in \C^\times$. Then, we define $\mcY_i$ by the quotient of $\wtd{\mcY}_i$ by relations $S_{i,xq_i^2} = A_{i,xq_i}S_{i,x}$ for $x \in \C^\times$.
Note that $\wtd{\mcY}_i$ is also a free $\mcY$-module.
For $i \in I$, we define a linear map $\wtd{S}_i : \mcY \longrightarrow \wtd{\mcY}_i$ by
\begin{equation*}
    \wtd{S}_i :
    \xymatrixcolsep{3pc}\xymatrixrowsep{0pc}
    \xymatrix{
    \mcY \ \ar@{->}[r] & \wtd{\mcY}_i \\
    Y_{j, a} \ar@{|->}[r] & \delta_{ij} Y_{i, a} S_{i, a}
},
\end{equation*}
and extending to the whole ring $\mcY$ via the Leibniz rule $\wtd{S}_i (rs) = s \wtd{S}_i(r) + r \wtd{S}_i(s)$ for $r, s \in \mcY$.
Note that one can check that $\wtd{S}_i(Y_{j,a}^{-1}) = -\delta_{ij} Y_{i,a}^{-1}S_{i,a}$.

Let $S_i : \mcY \rightarrow \mcY_i$ be the composition of $\wtd{S}_i$ and the canonical projection $\wtd{\mcY}_i \rightarrow \mcY_i$, which is called the {\it $i$-th screening operator}. Then the image of $q$-character homomorphism is given as follows: 

\begin{thm}\!\!\!{\em \cite{FR99, FM01, He04}} \label{thm:Image of qchar is the intersection of kernels of Si}
We have
\begin{equation*}
    {\rm Im} (\chi_q) = \bigcap_{i \in I} \, {\rm ker} \, (S_i),
\end{equation*}
where ${\rm ker}\, (S_i)$ is the kernel of $S_i$ realized as
\begin{equation} \label{eq:description of ker Si}
    {\rm ker}\, (S_i) = \Z[Y_{j,a}]_{j \neq i, a \in \C^\times} \otimes \Z[Y_{i,b} + Y_{i,b}A_{i,bq_i}^{-1}]_{b \in \C^\times}.
\end{equation}
In particular, a non-zero element in ${\rm Im} (\chi_q)$ has at least one dominant monomial.
\end{thm}

\section{Path description for $q$-characters of fundamental modules} \label{sec:path description}
\noindent
From now on, we assume that $\g$ is of type $C_n$. Here we use the simple roots that are numbered as shown on the following Dynkin diagram:
\begin{equation*}
    \begin{split}
    \begin{tikzpicture}[scale=0.5]
    \draw (-1,0) node[anchor=east] {\footnotesize $C_{n}$ : };
    \draw (0 cm,0) -- (3.5 cm,0);
    \draw (4.5 cm,0) -- (6 cm,0);
    \draw[dotted] (3.5 cm, 0) -- (4.5 cm, 0);
    \draw (6 cm, 0.1 cm) -- +(2 cm,0);
    \draw (6 cm, -0.1 cm) -- +(2 cm,0);
    \draw[shift={(6.8, 0)}, rotate=180] (135 : 0.45cm) -- (0,0) -- (-135 : 0.45cm);
    \draw[fill=white] (0 cm, 0 cm) circle (.25cm) node[below=4pt]{\scr{$1$}};
    \draw[fill=white] (2 cm, 0 cm) circle (.25cm) node[below=4pt]{\scr{$2$}};
    \draw[fill=white] (6 cm, 0 cm) circle (.25cm) node[below=4pt]{\scr{$n-1$}};
    \draw[fill=white] (8 cm, 0 cm) circle (.25cm) node[below=4pt]{\scr{$n$}};
    \end{tikzpicture}
    \end{split}
\end{equation*}

\subsection{Monomials of paths}
By abuse of notation, we write 
\begin{equation} \label{eq:Y and A}
    Y_{i,k} = Y_{i,q^k}, \quad 
    A_{i,k} = A_{i,q^k},
\end{equation}
for $i \in I$ and $k \in \Z_+$ (recall \eqref{eq:variable A}).

We define a map $\m $ from $\Pa $ into $\Z[Y_{\ov{j},k}^{\pm 1}]_{(j,k) \in \X}$ to associate a path with a monomial as follows:
\begin{equation} \label{eq:monomial map}
\begin{split}
    \xymatrixcolsep{2pc}\xymatrixrowsep{0pc}
    \xymatrix{
        \mathsf{m}  : \Pa  \ar@{->}[r] & \displaystyle \mathbb{Z}\left[Y_{\ov{j}, \ell}^{\pm 1}\right]_{(j, \ell) \in \X} \\
        \qquad p \ar@{|->}[r] & \displaystyle \prod_{(j, \ell) \in C_{p,\!+} } Y_{\ov{j},\ell+2\delta(j > n)} \prod_{(j, \ell) \in C_{p,\!-} } Y_{\ov{j}, \ell+2\delta(j \ge n)}^{-1}  \hspace{-1.5cm}
    }
\end{split}
\end{equation}

\subsubsection{Admissible paths}
Let us define a subset $\ov{\Pa}_{i,k} $ of $\Pa_{i,k} $ by 
\begin{equation} \label{eq:admissible paths}
	\ov{\Pa}_{i,k}  =
    \left\{\, 
	(j, \ell_j)_{j \in \I} \in \Pa_{i,k}  \, | \, \ell_j \le \ell_{N-j} \, \text{\,for $j \in I$} 
	\,\right\}.
\end{equation}
We call a path in $\ov{\Pa}_{i,k}$ an {\it admissible path} (of type C).

\begin{lem} \label{lem:unqiue dominant monomial}
    There exists a unique dominant $($resp.~antidominant$)$ monomial $Y_{i,k}$ $($resp.~$Y_{i,k+2n}^{-1}$$)$ in $\m\left(\ov{\Pa}_{i,k}\right)$.
\end{lem}
\begin{proof}
    Let $p \in \ov{\Pa}_{i,k}$. 
    First, if $p$ is the highest (resp.~lowest) path (cf.~Example \ref{ex:highest and lowest corners}), then $\m(p) = Y_{i,k}$ (resp.~$Y_{i,k+2n}^{-1}$).
    Next, we claim that if $p$ is not the highest (resp.~lowest) path, then $\m(p)$ cannot be $Y_{i,k}$ (resp.~$Y_{i,k+2n}^{-1}$).
    We only prove the case that $p$ is not the highest path. 
    The lowest case is similar.
    \smallskip
    
    Let us assume that $p$ is not the highest path. By Remark \ref{rem:upper and lower corners}, there exist corners $(j_1, \ell_1) \in C_{p,+}$ and $(j_2, \ell_2) \in C_{p,-}$. Note that $j_1 \neq j_2$ by definition.
	\smallskip
	
	\noindent
	{\it Case 1}. $\ov{j_1} = \ov{j_2}$.
    In this case, $\m(p)$ contains a monomial as its factor, which is not equal to $1$, given as follows:
    \begin{equation*}
    \begin{cases}
        Y_{\ov{j_1},\ell_1} Y_{\ov{j_2},\ell_2+2}^{-1} & \text{if $j_1 < n < j_2$,} \\
        Y_{\ov{j_1},\ell_1+2} Y_{\ov{j_2},\ell_2}^{-1} & \text{if $j_2 < n < j_1$,}
    \end{cases}
    \end{equation*}
    where $\ell_1 \le \ell_2$ (resp.~$\ell_2 \le \ell_1$) in the first (resp.~second) case by definition of $\ov{\Pa}_{i,k}$.
    Thus the monomial $\m(p)$ cannot be $Y_{i,k}$, since the above factor $(\neq 1)$ cannot be cancelled out by other factors.
    \smallskip
    
    \noindent
    {\it Case 2}. $\ov{j_1} \neq \ov{j_2}$.
    If one can find a pair of $(j_1',\ell_1') \in C_{p,+}$ and $(j_2', \ell_2') \in C_{p,-}$ such that $\ov{j_1'} = \ov{j_2'}$ (possibly, from $(j_1,\ell_1)$ and $(j_2,\ell_2)$), then it is done by {\it Case 1}. Otherwise, $\m(p)$ cannot be $Y_{i,k}$ by definition of $\m$.
\end{proof}

\begin{lem} \label{lem:thin property}
    The induced map $\m|_{\ov{\Pa}_{i,k}}$ from $\m$ is injective. 
\end{lem}
\begin{proof}
    Let $p_1$ and $p_2$ be paths in $\ov{\Pa}_{i,k}$.
    Assume that $\m(p_1) = \m(p_2)$.
    If $p_1$ is the highest or lowest path in $\ov{\Pa}_{i,k}$, so is $p_2$ by Lemma \ref{lem:unqiue dominant monomial}.
    Assume contrary, that is, both $p_1$ and $p_2$ are not the highest or lowest paths.
    
    If $p_1 \neq p_2$, then it follows from Lemma \ref{lem:characterization of paths} that there exists an upper (resp.~lower) corner of $p_1$, say $(j,\ell)$, which is not an upper (resp.~lower) corner of $p_2$.
	If $j = n$, then $\m(p_1)$ has $Y_{n,\ell}^{\pm 1}$ as its factor, while $\m(p_2)$ does not. Thus $j \neq n$, so we may assume $j < n$ (The case of $j > n$ also give a contraction by similar argument as shown below). 

    Let us take the largest $j < n$ such that $(j,\ell)$ is a corner of $p_1$.
    Since $\m(p_1) = \m(p_2)$, the path $p_2$ should have a corner at $(N-j, \ell-2)$. 
    Then $p_2$ has no corner at $(j,\ell-2)$ since, otherwise, $p_1$ has a corner at $(N-j, \ell-4)$ (recall $\m(p_1) = \m(p_2)$). This is a contradiction to the admissible condition \eqref{eq:admissible paths}.
    Therefore, the part of $p_2$ adjacent to $(j,\ell-2)$ should be one of the following configurations:
    
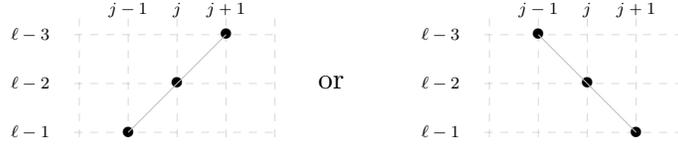
\begin{equation*}
        \begin{split}
            \begin{tikzpicture}[scale=0.65]
                \draw[help lines, color=gray!30, dashed] (-2.1,-0.1) grid (2.1, 2.5); 
                \node at (-1, 2.5) {\scalebox{0.8}{\scr{$j-1$}}};
                \node at (0, 2.5) {\scalebox{0.8}{\scr{$j$}}};
                \node at (1, 2.5) {\scalebox{0.8}{\scr{$j+1$}}};
                %
                \node at (-3, 2) {\scalebox{0.8}{\scr{$\ell-3$}}};
                \node at (-3, 1) {\scalebox{0.8}{\scr{$\ell-2$}}};
                \node at (-3, 0) {\scalebox{0.8}{\scr{$\ell-1$}}};
                \node (1) at (0, 1) {$\bullet$};
                \node (2) at (-1, 0) {$\bullet$};
                \node (3) at (1,  2) {$\bullet$};
                %
                \draw[draw=lightgray] (1.center) -- (2.center);
                \draw[draw=lightgray] (1.center) -- (3.center);
            \end{tikzpicture}
            & \quad \raisebox{0.8cm}{\text{ or }} \qquad
            \begin{tikzpicture}[scale=0.65]
                \draw[help lines, color=gray!30, dashed] (-2.1,-0.1) grid (2.1, 2.5); 
                \node at (-1, 2.5) {\scalebox{0.8}{\scr{$j-1$}}};
                \node at (0, 2.5) {\scalebox{0.8}{\scr{$j$}}};
                \node at (1, 2.5) {\scalebox{0.8}{\scr{$j+1$}}};
                %
                \node at (-3, 2) {\scalebox{0.8}{\scr{$\ell-3$}}};
                \node at (-3, 1) {\scalebox{0.8}{\scr{$\ell-2$}}};
                \node at (-3, 0) {\scalebox{0.8}{\scr{$\ell-1$}}};
                \node (1) at (0, 1) {$\bullet$};
                \node (2) at (-1, 2) {$\bullet$};
                \node (3) at (1, 0) {$\bullet$};
                %
                \draw[draw=lightgray] (1.center) -- (2.center);
                \draw[draw=lightgray] (1.center) -- (3.center);
            \end{tikzpicture}
        \end{split}
    \end{equation*}
    In any case, there is a corner at $(j',\ell')$ in $p_2$ and, therefore, also in $p_1$ with $j < j' < n$. This contradicts our choice of $j$.
    Hence we have $p_1 = p_2$ when $\m(p_1) = \m(p_2)$.
\end{proof}

\subsection{Connected components of $\ov{\Pa}_{i,k}$}
For $p_1,\, p_2 \in \ov{\Pa}_{i,k}$, we say that the paths $p_1$ and $p_2$ are {\it connected} if one can obtained from another one by applying a (finite) sequence of raising or lowering moves in Section \ref{subsec:moves}.
Then a {\it connected component} is a subset of $\ov{\Pa}_{i,k}$ consisting of paths in $\ov{\Pa}_{i,k}$, which are connected to each other.
In particular, for $j \in I$, 
we say that two paths $p_1$ and $p_2$ are in the {\it $j$-connected component} or {\it $j$-component} for short, if $p_1$ and $p_2$ are connected by a sequence consisting of raising or lowering moves associated with $k \in \Io$ such that $\ov{k} = j$.
Note that $j$-connected component may consist of only one path.
For convenience, we write $p_1 \overset{j}{\rightarrow} p_2$ (resp.~$p_1 \overset{\,\,j}{\leftarrow} p_2$) if $p_2$ is obtained from $p_1$ by applying a lowering (resp.~raising) move associated with $k \in \Io$ such that $\ov{k} = j$.

\begin{prop} \label{prop:disjoint union}
Any path in $\ov{\Pa}_{i,k}$ is contained in at least one $j$-component for $j \in I$. Moreover, for a fixed $j \in I$, $\ov{\Pa}_{i,k}$ can be written as a disjoint union of $j$-components.
\end{prop}
\begin{proof}
The first assertion follows directly from Lemma \ref{lem:characterization of paths}, and the second assertion is straightforward.
\end{proof}

\begin{ex} \label{ex:j-component}
{\em 
Let $n = 3$ and let $j = 2$.
The set consisting of following paths is a $2$-component of $\ov{\Pa}_{3,0}$:
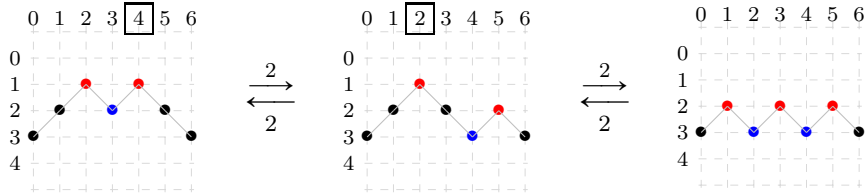
\begin{equation*}
\begin{tikzpicture}[scale=0.35, baseline=(current  bounding  box.center)]
            \draw[help lines, color=gray!30, dashed] (-3.1,-3.1) grid (3.1, 3.5); 
            %
            \node at (-3, 3.5) {\scr{$0$}};
            \node at (-2, 3.5) {\scr{$1$}};
            \node at (-1, 3.5) {\scr{$2$}};
            \node at (0, 3.5) {\scr{$3$}};
            \node at (1, 3.5) {\scr{$\boxed{4}$}};
            \node at (2, 3.5) {\scr{$5$}};
            \node at (3, 3.5) {\scr{$6$}};
            %
            \node at (-3.7, 2) {\scr{$0$}};
            \node at (-3.7, 1) {\scr{$1$}};
            \node at (-3.7, 0) {\scr{$2$}};
            \node at (-3.7, -1) {\scr{$3$}};
            \node at (-3.7, -2) {\scr{$4$}};
            %
            \node (0) at (-3,   -1) {$\bullet$};
            \node (1) at (-2,   0) {$\bullet$};
            \node (2) at (-1,   1) {\red{$\bullet$}};
            \node (3) at ( 0,   0) {\blue{$\bullet$}};
            \node (4) at ( 1,   1) {\red{$\bullet$}};
            \node (5) at ( 2,   0) {$\bullet$};
            \node (6) at ( 3,   -1) {$\bullet$};
            %
            \draw[draw=lightgray] (0.center) -- (1.center);
            \draw[draw=lightgray] (1.center) -- (2.center);
            \draw[draw=lightgray] (2.center) -- (3.center);
            \draw[draw=lightgray] (3.center) -- (4.center);
            \draw[draw=lightgray] (4.center) -- (5.center);
            \draw[draw=lightgray] (5.center) -- (6.center);
	\end{tikzpicture}
	\quad
	\overset{2}{\underset{\scalebox{1.1}{$\underset{2}{\longleftarrow}
	$}}{\longrightarrow}}
	\quad
	\begin{tikzpicture}[scale=0.35, baseline=(current  bounding  box.center)]
            \draw[help lines, color=gray!30, dashed] (-3.1,-3.1) grid (3.1, 3.5); 
            %
            \node at (-3, 3.5) {\scr{$0$}};
            \node at (-2, 3.5) {\scr{$1$}};
            \node at (-1, 3.5) {\scr{$\boxed{2}$}};
            \node at (0, 3.5) {\scr{$3$}};
            \node at (1, 3.5) {\scr{$4$}};
            \node at (2, 3.5) {\scr{$5$}};
            \node at (3, 3.5) {\scr{$6$}};
            %
            \node at (-3.7, 2) {\scr{$0$}};
            \node at (-3.7, 1) {\scr{$1$}};
            \node at (-3.7, 0) {\scr{$2$}};
            \node at (-3.7, -1) {\scr{$3$}};
            \node at (-3.7, -2) {\scr{$4$}};
            %
            \node (0) at (-3,   -1) {$\bullet$};
            \node (1) at (-2,   0) {$\bullet$};
            \node (2) at (-1,   1) {\red{$\bullet$}};
            \node (3) at ( 0,   0) {$\bullet$};
            \node (4) at ( 1,   -1) {\blue{$\bullet$}};
            \node (5) at ( 2,   0) {\red{$\bullet$}};
            \node (6) at ( 3,   -1) {$\bullet$};
            %
            \draw[draw=lightgray] (0.center) -- (1.center);
            \draw[draw=lightgray] (1.center) -- (2.center);
            \draw[draw=lightgray] (2.center) -- (3.center);
            \draw[draw=lightgray] (3.center) -- (4.center);
            \draw[draw=lightgray] (4.center) -- (5.center);
            \draw[draw=lightgray] (5.center) -- (6.center);
	\end{tikzpicture}
	\quad
	\overset{2}{\underset{\scalebox{1.1}{$\underset{2}{\longleftarrow}
	$}}{\longrightarrow}}
	\quad
		\begin{tikzpicture}[scale=0.35, baseline=(current  bounding  box.center)]
            \draw[help lines, color=gray!30, dashed] (-3.1,-3.1) grid (3.1, 3.5); 
            %
            \node at (-3, 3.5) {\scr{$0$}};
            \node at (-2, 3.5) {\scr{$1$}};
            \node at (-1, 3.5) {\scr{$2$}};
            \node at (0, 3.5) {\scr{$3$}};
            \node at (1, 3.5) {\scr{$4$}};
            \node at (2, 3.5) {\scr{$5$}};
            \node at (3, 3.5) {\scr{$6$}};
            %
            \node at (-3.7, 2) {\scr{$0$}};
            \node at (-3.7, 1) {\scr{$1$}};
            \node at (-3.7, 0) {\scr{$2$}};
            \node at (-3.7, -1) {\scr{$3$}};
            \node at (-3.7, -2) {\scr{$4$}};
            %
            \node (0) at (-3,   -1) {$\bullet$};
            \node (1) at (-2,   0) {\red{$\bullet$}};
            \node (2) at (-1,   -1) {\blue{$\bullet$}};
            \node (3) at ( 0,   0) {\red{$\bullet$}};
            \node (4) at ( 1,   -1) {\blue{$\bullet$}};
            \node (5) at ( 2,   0) {\red{$\bullet$}};
            \node (6) at ( 3,   -1) {$\bullet$};
            %
            \draw[draw=lightgray] (0.center) -- (1.center);
            \draw[draw=lightgray] (1.center) -- (2.center);
            \draw[draw=lightgray] (2.center) -- (3.center);
            \draw[draw=lightgray] (3.center) -- (4.center);
            \draw[draw=lightgray] (4.center) -- (5.center);
            \draw[draw=lightgray] (5.center) -- (6.center);
	\end{tikzpicture}
\end{equation*}
Here the lowering and raising moves occur under the boxed numbers.
}
\end{ex}

\begin{rem} \label{rem:moves do not preserve admissibility}
{\em 
In general, the raising and lowering moves do not preserve admissible paths. In Example \ref{ex:j-component}, if we apply the lowering move associated with $2$ to the first path, then the resulting path is not an admissible path.
Therefore, if there are two upper (resp.~lower) corners at $(j,\ell)$ and $(N-j,\ell)$, then we apply the lowering (resp.~raising) move to the corner at $(N-j,\ell)$ (resp.~$(j,\ell)$) as in Example \ref{ex:j-component}.
}
\end{rem}

\begin{lem} \label{lem:maximal length of j-component}
Let $j \in I$, and let $C$ be a $j$-component of $\ov{\Pa}_{i,k}$. Then the number of paths in $C$ is at most $4$.
\end{lem}
\begin{proof}
Let $p \in C$.
By Lemma \ref{lem:characterization of paths}, it is enough to find all distinct paths in the $j$-component $C$ from $p$ by applying raising or lowering moves associated with $k \in \Io$ such that $\ov{k} = j$, as possible.
For this, we denote by $C_{p,\pm}^{(j)}$ the subset of $C_{p,\pm}$ consisting of upper and lower corners at $(k,\ell)$ for $\ov{k} = j$ and $\ell \in \Z$, respectively.
By definition, the possibility of $(|C_{p,+}^{(j)}|, |C_{p,-}^{(j)}|)$ is $(0,0)$, $(1,0)$, $(0,1)$, $(2, 0)$, $(1, 1)$ or $(0,2)$.
In each case, we will characterize $C$ as follows:
\smallskip

\noindent
{\it Case 1}. $(|C_{p,+}^{(j)}|, |C_{p,-}^{(j)}|) = (0, 0)$. In this case, $C = \left\{\, p \,\right\}$.

\noindent
{\it Case 2}. $(|C_{p,+}^{(j)}|, |C_{p,-}^{(j)}|) = (1, 0)$ or $(0, 1)$. 
In this case, we claim that
\begin{equation} \label{eq:C=2 case}
	C = 
	\begin{cases}
		\left\{\, p,\, p\A_{k,\ell+1}^{-1} \,\right\} & \text{if $(|C_{p,+}^{(j)}|, |C_{p,-}^{(j)}|) = (1, 0)$,} \\
		\left\{\,\, p\A_{k,\ell-1},\, p \,\,\right\} & \text{if $(|C_{p,+}^{(j)}|, |C_{p,-}^{(j)}|) = (0, 1)$,}
	\end{cases}
\end{equation}
where $\A_{k,\ell\pm 1}^{\pm 1}$ is given in Section \ref{subsec:moves}.
We verify our claim in the case of $(|C_{p,+}^{(j)}|, |C_{p,-}^{(j)}|) = (1, 0)$. The other case is almost identical. 

Let $(k,\ell)$ be the upper corner of $p$ with $\ov{k} = j$. 
If $k = n$, then it is done, since there is no restriction to apply the lowering move to $p$ in $\ov{\Pa}_{i,k}$. 
Assume that $k \neq n$.
If $k > n$, then $(j,\ell_j)$ is neither upper nor lower corner and $\ell_j \le \ell$ by \eqref{eq:admissible paths}. Thus we can apply the lowering move to $p$, and then we have \eqref{eq:C=2 case}.
Suppose $j = k < n$. In this case, the part of $p$ adjacent to $(\ov{k},\ell')$ is one of the following configurations:
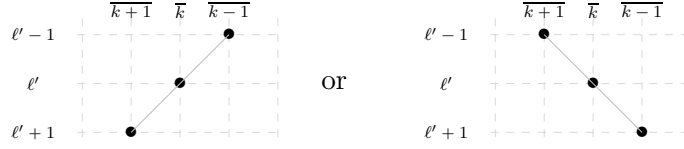
\begin{equation*}
        \begin{split}
            \begin{tikzpicture}[scale=0.65]
                \draw[help lines, color=gray!30, dashed] (-2.1,-0.1) grid (2.1, 2.5); 
                \node at (-1, 2.5) {\scalebox{0.8}{\scr{$\ov{k+1}$}}};
                \node at (0, 2.5) {\scalebox{0.8}{\scr{$\ov{k}$}}};
                \node at (1, 2.5) {\scalebox{0.8}{\scr{$\ov{k-1}$}}};
                %
                \node at (-3, 2) {\scalebox{0.8}{\scr{$\ell'-1$}}};
                \node at (-3, 1) {\scalebox{0.8}{\scr{$\ell'$}}};
                \node at (-3, 0) {\scalebox{0.8}{\scr{$\ell'+1$}}};
                \node (1) at (0, 1) {$\bullet$};
                \node (2) at (-1, 0) {$\bullet$};
                \node (3) at (1,  2) {$\bullet$};
                %
                \draw[draw=lightgray] (1.center) -- (2.center);
                \draw[draw=lightgray] (1.center) -- (3.center);
            \end{tikzpicture}
            & \quad \raisebox{0.8cm}{\text{ or }} \qquad
            \begin{tikzpicture}[scale=0.65]
                \draw[help lines, color=gray!30, dashed] (-2.1,-0.1) grid (2.1, 2.5); 
                \node at (-1, 2.5) {\scalebox{0.8}{\scr{$\ov{k+1}$}}};
                \node at (0, 2.5) {\scalebox{0.8}{\scr{$\ov{k}$}}};
                \node at (1, 2.5) {\scalebox{0.8}{\scr{$\ov{k-1}$}}};
                %
                \node at (-3, 2) {\scalebox{0.8}{\scr{$\ell'-1$}}};
                \node at (-3, 1) {\scalebox{0.8}{\scr{$\ell'$}}};
                \node at (-3, 0) {\scalebox{0.8}{\scr{$\ell'+1$}}};
                \node (1) at (0, 1) {$\bullet$};
                \node (2) at (-1, 2) {$\bullet$};
                \node (3) at (1, 0) {$\bullet$};
                %
                \draw[draw=lightgray] (1.center) -- (2.center);
                \draw[draw=lightgray] (1.center) -- (3.center);
            \end{tikzpicture}
        \end{split}
    \end{equation*}
	However, in any case, $\ell \le \ell' \le \ell+1$ is not possible by \eqref{eq:admissible paths}, since $(k,\ell)$ is the upper corner.
	Hence, $\ell + 2 \le \ell'$ which implies that we can apply the lowering move to $p$.
	
	\noindent
	{\it Case 3}. $(|C_{p,+}^{(j)}|, |C_{p,-}^{(j)}|) = (2, 0)$, $(1,1)$ or $(0, 2)$. First, we consider the case of $(|C_{p,+}^{(j)}|, |C_{p,-}^{(j)}|) = (2, 0)$. We denote by $(j_1, \ell_1)$ and $(j_2, \ell_2)$ the upper corners in $C_{p,+}^{(j)}$. Note that $j_1 \neq j_2$, so we assume that $j_1 < n < j_2$.
	If $\ell_1 = \ell_2 = \ell$, then it follows from Remark \ref{rem:moves do not preserve admissibility} that
	\begin{equation*}
		C = \left\{\, p,\, p\A_{j_2,\ell+1}^{-1},\, p\A_{j_2,\ell+1}^{-1}\A_{j_1,\ell+1}^{-1} \,\right\}.
	\end{equation*}
	If $\ell_1 < \ell_2$, then \eqref{eq:X} implies that $\ell_1 + 2 \le \ell_2$. Hence, we have
	\begin{equation*}
		C = \left\{\, p,\, p\A_{j_1,\ell_1+1}^{-1},\, p\A_{j_2,\ell_2+1}^{-1},\, p\A_{j_1,\ell_1+1}^{-1}\A_{j_2,\ell_2+1}^{-1} \,\right\}.
	\end{equation*}
	Here one can check that $p\A_{j_1,\ell_1+1}^{-1}\A_{j_2,\ell_2+1}^{-1} =  p\A_{j_2,\ell_2+1}^{-1}\A_{j_1,\ell_1+1}^{-1}$.
	By symmetry, we also conclude that $C$ has at most four paths in the case of $(|C_{p,+}^{(j)}|, |C_{p,-}^{(j)}|) = (0,2)$.
	Next, we consider the case of $(|C_{p,+}^{(j)}|, |C_{p,-}^{(j)}|) = (1,1)$.
	Let $(j_1, \ell_1) \in C_{p,+}^{(j)}$ and $(j_2, \ell_2) \in C_{p,-}^{(j)}$.
	Clearly, $j_1 \neq n$ and $j_2 \neq n$.
	The part of $p$ adjacent to $(j_1, \ell_1)$ and $(j_2, \ell_2)$ as follows:
	
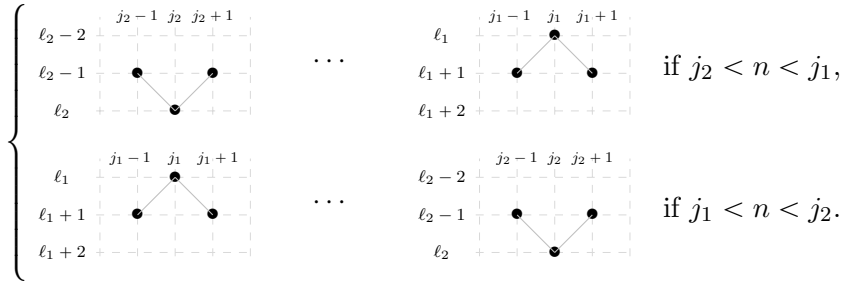
\begin{equation*}
	\begin{cases}
			\begin{tikzpicture}[scale=0.5]
                \draw[help lines, color=gray!30, dashed] (-2.1,-1.1) grid (2.1, 1.5); 
                \node at (-1, 1.5) {\scalebox{0.7}{\scr{$j_2-1$}}};
                \node at (0, 1.5) {\scalebox{0.7}{\scr{$j_2$}}};
                \node at (1, 1.5) {\scalebox{0.7}{\scr{$j_2+1$}}};
                %
                \node at (-3, 1) {\scalebox{0.8}{\scr{$\ell_2-2$}}};
                \node at (-3, 0) {\scalebox{0.8}{\scr{$\ell_2-1$}}};
                \node at (-3, -1) {\scalebox{0.8}{\scr{$\ell_2$}}};
                \node (1) at (0, -1) {$\bullet$};
                \node (2) at (-1, 0) {$\bullet$};
                \node (3) at (1,  0) {$\bullet$};
                %
                \draw[draw=lightgray] (1.center) -- (2.center);
                \draw[draw=lightgray] (1.center) -- (3.center);
            \end{tikzpicture}
            \qquad
            \raisebox{8mm}{$\cdots$}
            \qquad
	        \begin{tikzpicture}[scale=0.5]
            \draw[help lines, color=gray!30, dashed] (-2.1,-1.1) grid (2.1, 1.5); 
            \node at (-1, 1.5) {\scalebox{0.7}{\scr{$j_1-1$}}\,\,\,};
            \node at (0, 1.5) {\scalebox{0.7}{\scr{$j_1$}}};
            \node at (1, 1.5) {\scalebox{0.7}{\,\,\,\scr{$j_1+1$}}};
            %
            \node at (-3, 1) {\scalebox{0.8}{\scr{$\ell_1$}}};
            \node at (-3, 0) {\scalebox{0.8}{\scr{$\ell_1+1$}}};
            \node at (-3, -1) {\scalebox{0.8}{\scr{$\ell_1+2$}}};
            \node (1) at (0, 1) {$\bullet$};
            \node (2) at (-1, 0) {$\bullet$};
            \node (3) at (1,  0) {$\bullet$};
            %
            \draw[draw=lightgray] (1.center) -- (2.center);
            \draw[draw=lightgray] (1.center) -- (3.center);
        \end{tikzpicture}
        & \raisebox{7mm}{\text{if $j_2 < n < j_1$,}} \\ 
        \begin{tikzpicture}[scale=0.5]
            \draw[help lines, color=gray!30, dashed] (-2.1,-1.1) grid (2.1, 1.5); 
            \node at (-1, 1.5) {\scalebox{0.7}{\scr{$j_1-1$}}\,\,\,};
            \node at (0, 1.5) {\scalebox{0.7}{\scr{$j_1$}}};
            \node at (1, 1.5) {\scalebox{0.7}{\,\,\,\scr{$j_1+1$}}};
            %
            \node at (-3, 1) {\scalebox{0.8}{\scr{$\ell_1$}}};
            \node at (-3, 0) {\scalebox{0.8}{\scr{$\ell_1+1$}}};
            \node at (-3, -1) {\scalebox{0.8}{\scr{$\ell_1+2$}}};
            \node (1) at (0, 1) {$\bullet$};
            \node (2) at (-1, 0) {$\bullet$};
            \node (3) at (1,  0) {$\bullet$};
            %
            \draw[draw=lightgray] (1.center) -- (2.center);
            \draw[draw=lightgray] (1.center) -- (3.center);
        \end{tikzpicture}
        \qquad
        \raisebox{8mm}{$\cdots$}
        \qquad
        \begin{tikzpicture}[scale=0.5]
                \draw[help lines, color=gray!30, dashed] (-2.1,-1.1) grid (2.1, 1.5); 
                \node at (-1, 1.5) {\scalebox{0.7}{\scr{$j_2-1$}}};
                \node at (0, 1.5) {\scalebox{0.7}{\scr{$j_2$}}};
                \node at (1, 1.5) {\scalebox{0.7}{\scr{$j_2+1$}}};
                %
                \node at (-3, 1) {\scalebox{0.8}{\scr{$\ell_2-2$}}};
                \node at (-3, 0) {\scalebox{0.8}{\scr{$\ell_2-1$}}};
                \node at (-3, -1) {\scalebox{0.8}{\scr{$\ell_2$}}};
                \node (1) at (0, -1) {$\bullet$};
                \node (2) at (-1, 0) {$\bullet$};
                \node (3) at (1,  0) {$\bullet$};
                %
                \draw[draw=lightgray] (1.center) -- (2.center);
                \draw[draw=lightgray] (1.center) -- (3.center);
            \end{tikzpicture}
            & \raisebox{7mm}{\text{if $j_1 < n < j_2$.}}
    \end{cases} 
	\end{equation*}
	Here it follows from \eqref{eq:admissible paths} that $\ell_2 \le \ell_1$ if $j_2 < n < j_1$, and $\ell_1 +2 \le \ell_2$ if $j_1 < n < j_2$.
	Hence we have
	\begin{equation*}
		C = 
		\begin{cases}
			\left\{\, p\A_{j_2,\ell_2-1},\, p,\, p\A_{j_1,\ell_1+1}^{-1} \,\right\} & \text{if $j_2 < n < j_1$,} \\
			\left\{\, p\A_{j_1,\ell_1+1}^{-1},\, p,\, p\A_{j_2,\ell_2-1}^{-1} \,\right\} & \text{if $j_1 < n < j_2$.}
		\end{cases}
	\end{equation*}

	By {\it Case 1}--{\it Case 3}, we have seen that $C$ has at most four paths. This completes the proof.
\end{proof}

\subsection{Path description for $q$-characters of fundamental modules}
Now, we are in a position to state the main result of this paper.
\begin{thm} \label{thm:main1}
    For $i \in I$ and $k \in \Z_+$, we have
    \begin{equation} \label{eq:path model}
    \begin{split}
    	\chi_q(L(Y_{i, k}))
		=
		\sum_{p\, \in\, \ov{\Pa}_{i,k} } \m (p),
	\end{split}
    \end{equation}
    where $\ov{\Pa}_{i,k}$ is given in \eqref{eq:admissible paths}.
\end{thm}
\begin{proof}
Let $j \in I$ be given.
By Proposition \ref{prop:disjoint union}, one can write
\begin{equation*}
    \ov{\Pa}_{i,k} = \bigsqcup_{t=1}^M C_t,
\end{equation*}
where $C_t$ is a $j$-component of $\ov{\Pa}_{i,k}$ for $1 \le t \le M$.
Put
\begin{equation*}
    \chi =
    \sum_{p\, \in\, \ov{\Pa}_{i,k} } \m (p)
    \quad
    \text{ and }
    \quad\!
    \m(C_t) = \sum_{p \in C_t} \m(p).
\end{equation*}
Then we have
\begin{equation*}
    \chi = \sum_{t=1}^M \m(C_t).
\end{equation*}
By Lemma \ref{lem:thin property}, any monomial in $\m(C_t)$ is not overlapped with other monomials in $\m(C_{t'})$ for $t' \neq t$.

Now, we claim that $\chi \in \ker(S_j)$ for all $j \in I$, which implies 
$\chi = \chi_q(L(Y_{i,k}))$
by Theorem \ref{thm:Image of qchar is the intersection of kernels of Si} and Lemma \ref{lem:unqiue dominant monomial}.

For this, let us consider a $j$-component $C$ of $\ov{\Pa}_{i,k}$.
We will show that $\m(C) \in \ker(S_j)$ which implies our claim, since $S_j$ is linear by its definition.
By Lemma \ref{lem:maximal length of j-component}, it is enough to consider four cases as follows:
\smallskip

\noindent
{\it Case 1}. $|C|=1$. The path $p$ in $\ov{\Pa}_{i,k} \cap C$ has no corner at $(k,\ell)$ with $\ov{k} = j$ for any $\ell \in \Z$. This implies that $\m(p)$ has no factor $Y_{j,\ell}^{\pm 1}$. Hence, $\m(p) \in \ker(S_j)$.
\smallskip

\noindent
{\it Case 2}. $|C|=2$. Let $p_1$ and $p_2$ be the paths in $\ov{\Pa}_{i,k} \cap C$.
We may assume that $p_2 = p_1 \A_{k,\ell}^{-1}$. 
Then the local configurations of $p_1$ and $p_2$ near by $(k,\ell)$ with $\ov{k} = j$ are depicted as shown below:
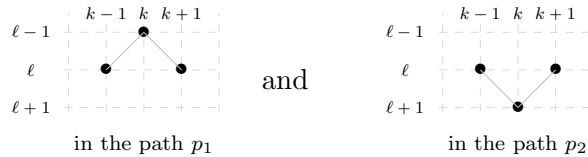
\begin{equation*}
    \begin{split}
        \begin{tikzpicture}[scale=0.5]
            \draw[help lines, color=gray!30, dashed] (-2.1,-1.1) grid (2.1, 1.5); 
            \node at (-1, 1.5) {\scalebox{0.8}{\scr{$k-1$}}};
            \node at (0, 1.5) {\scalebox{0.8}{\scr{$k$}}};
            \node at (1, 1.5) {\scalebox{0.8}{\scr{$k+1$}}};
            %
            \node at (-3, 1) {\scalebox{0.8}{\scr{$\ell-1$}}};
            \node at (-3, 0) {\scalebox{0.8}{\scr{$\ell$}}};
            \node at (-3, -1) {\scalebox{0.8}{\scr{$\ell+1$}}};
            \node (1) at (0, 1) {$\bullet$};
            \node (2) at (-1, 0) {$\bullet$};
            \node (3) at (1,  0) {$\bullet$};
            %
            \draw[draw=lightgray] (1.center) -- (2.center);
            \draw[draw=lightgray] (1.center) -- (3.center);
            \node at (0, -2) {\scr{\text{in the path $p_1$}}};
        \end{tikzpicture}
        & \quad \raisebox{1cm}{\text{ and }} \qquad
        \begin{tikzpicture}[scale=0.5]
            \draw[help lines, color=gray!30, dashed] (-2.1,-1.1) grid (2.1, 1.5); 
            \node at (-1, 1.5) {\scalebox{0.8}{\scr{$k-1$}}};
            \node at (0, 1.5) {\scalebox{0.8}{\scr{$k$}}};
            \node at (1, 1.5) {\scalebox{0.8}{\scr{$k+1$}}};
            %
            \node at (-3, 1) {\scalebox{0.8}{\scr{$\ell-1$}}};
            \node at (-3, 0) {\scalebox{0.8}{\scr{$\ell$}}};
            \node at (-3, -1) {\scalebox{0.8}{\scr{$\ell+1$}}};
            \node (1) at (0, -1) {$\bullet$};
            \node (2) at (-1, 0) {$\bullet$};
            \node (3) at (1,  0) {$\bullet$};
            %
            \draw[draw=lightgray] (1.center) -- (2.center);
            \draw[draw=lightgray] (1.center) -- (3.center);
            \node at (0, -2) {\scr{\text{in the path $p_2$}}};
        \end{tikzpicture}
    \end{split}
\end{equation*}
where the other parts of $p_1$ and $p_2$ are same.
By \eqref{eq:monomial map}, we have
\begin{equation*}
    \m(p_1) + \m(p_2) = \left(Y_{\ov{k}, \ell-1 + 2\delta(k>n)} + Y_{\ov{k}, \ell-1 + \delta(k>n)} A_{\ov{k}, \ell + 2\delta(k>n)}^{-1} \right)M,
\end{equation*}
where $M \in \mathbb{Z}\left[Y_{\ov{l}, \ell}^{\pm 1}\right]_{(l, \ell) \in \X}$ has no factor $Y_{j,p}^{\pm 1}$ for $p \in \Z$ since $|C| = 2$.
Since the $j$-th screening operator $S_j$ is a derivation,
it follows from \eqref{eq:description of ker Si} that $\m(p_1) + \m(p_2) \in \ker(S_j)$.
\smallskip

\noindent
{\it Case 3}. $|C|=3$.
If $j = n$, then the number of paths in $C$ is at most $2$, since a path in $C$ has at most one upper corner.
Therefore, we assume $j \neq n$.
Let $p_1$, $p_2$ and $p_3$ be the paths in $\ov{\Pa}_{i,k} \cap C$.
Then we may assume that
\begin{equation*}
    p_2 = p_1 \A_{N-j,\ell}^{-1} \quad \text{ and } \quad
    p_3 = p_2 \A_{j, \ell}^{-1}.
\end{equation*}
By the proof of Lemma \ref{lem:maximal length of j-component}, the upper corners in $p_1$ should be located horizontally; otherwise, $|C|>3$ (this case  is in {\it Case 4} below).
Then the local configurations of $p_1$, $p_2$ and $p_3$ are depicted as follows:
{\allowdisplaybreaks
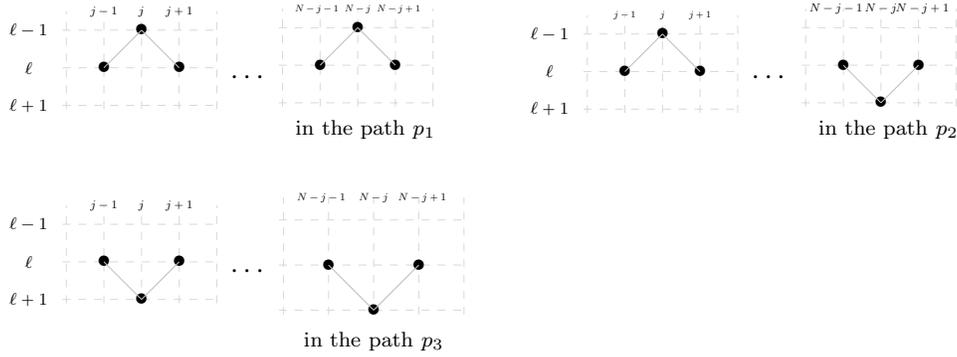
\begin{align*}
		\raisebox{0.5cm}{
        \begin{tikzpicture}[scale=0.5]
            \draw[help lines, color=gray!30, dashed] (-2.1,-1.1) grid (2.1, 1.5); 
            \node at (-1, 1.5) {\scalebox{0.55}{\scr{$j-1$}}};
            \node at (0, 1.5) {\scalebox{0.55}{\scr{$j$}}};
            \node at (1, 1.5) {\scalebox{0.55}{\scr{$j+1$}}};
            %
            \node at (-3, 1) {\scalebox{0.8}{\scr{$\ell-1$}}};
            \node at (-3, 0) {\scalebox{0.8}{\scr{$\ell$}}};
            \node at (-3, -1) {\scalebox{0.8}{\scr{$\ell+1$}}};
            \node (1) at (0, 1) {$\bullet$};
            \node (2) at (-1, 0) {$\bullet$};
            \node (3) at (1,  0) {$\bullet$};
            %
            \draw[draw=lightgray] (1.center) -- (2.center);
            \draw[draw=lightgray] (1.center) -- (3.center);
        \end{tikzpicture}
        }
        & \raisebox{1cm}{$\cdots$}
        \raisebox{0.15cm}{
        \begin{tikzpicture}[scale=0.5]
            \draw[help lines, color=gray!30, dashed] (-2.1,-1.1) grid (2.1, 1.5); 
            \node at (-1, 1.5) {\scalebox{0.45}{\scr{$N-j-1$}\,\,\,}};
            \node at (0, 1.5) {\scalebox{0.45}{\scr{$N-j$}}};
            \node at (1, 1.5) {\scalebox{0.45}{\,\,\,\scr{$N-j+1$}}};
            %
            \node (1) at (0, 1) {$\bullet$};
            \node (2) at (-1, 0) {$\bullet$};
            \node (3) at (1,  0) {$\bullet$};
            %
            \draw[draw=lightgray] (1.center) -- (2.center);
            \draw[draw=lightgray] (1.center) -- (3.center);
           	\node at (0.2, -1.7) {\scr{\text{in the path $p_1$}}};
        \end{tikzpicture}
        }
        &
        \raisebox{0.45cm}{
        \begin{tikzpicture}[scale=0.5]
            \draw[help lines, color=gray!30, dashed] (-2.1,-1.1) grid (2.1, 1.5); 
            \node at (-1, 1.5) {\scalebox{0.45}{\scr{$j-1$}}};
            \node at (0, 1.5) {\scalebox{0.45}{\scr{$j$}}};
            \node at (1, 1.5) {\scalebox{0.45}{\scr{$j+1$}}};
            %
            \node at (-3, 1) {\scalebox{0.8}{\scr{$\ell-1$}}};
            \node at (-3, 0) {\scalebox{0.8}{\scr{$\ell$}}};
            \node at (-3, -1) {\scalebox{0.8}{\scr{$\ell+1$}}};
            \node (1) at (0, 1) {$\bullet$};
            \node (2) at (-1, 0) {$\bullet$};
            \node (3) at (1,  0) {$\bullet$};
            %
            \draw[draw=lightgray] (1.center) -- (2.center);
            \draw[draw=lightgray] (1.center) -- (3.center);
        \end{tikzpicture}
        }
        & \raisebox{1cm}{$\cdots$}
        \raisebox{0.15cm}{
        \begin{tikzpicture}[scale=0.5]
            \draw[help lines, color=gray!30, dashed] (-2.1,-1.1) grid (2.1, 1.5); 
            \node at (-1, 1.5) {\scalebox{0.55}{\scr{$N-j-1$}}\,\,\,};
            \node at (0, 1.5) {\scalebox{0.55}{\scr{$N-j$}}};
            \node at (1, 1.5) {\scalebox{0.55}{\,\,\,\scr{$N-j+1$}}};
            %
            \node (1) at (0, -1) {$\bullet$};
            \node (2) at (-1, 0) {$\bullet$};
            \node (3) at (1,  0) {$\bullet$};
            %
            \draw[draw=lightgray] (1.center) -- (2.center);
            \draw[draw=lightgray] (1.center) -- (3.center);
            \node at (0.2, -1.7) {\scr{\text{in the path $p_2$}}};
        \end{tikzpicture}
        } \\
        \raisebox{0.5cm}{
        \begin{tikzpicture}[scale=0.5]
            \draw[help lines, color=gray!30, dashed] (-2.1,-1.1) grid (2.1, 1.5); 
            \node at (-1, 1.5) {\scalebox{0.55}{\scr{$j-1$}}};
            \node at (0, 1.5) {\scalebox{0.55}{\scr{$j$}}};
            \node at (1, 1.5) {\scalebox{0.55}{\scr{$j+1$}}};
            %
            \node at (-3, 1) {\scalebox{0.8}{\scr{$\ell-1$}}};
            \node at (-3, 0) {\scalebox{0.8}{\scr{$\ell$}}};
            \node at (-3, -1) {\scalebox{0.8}{\scr{$\ell+1$}}};
            \node (1) at (0, -1) {$\bullet$};
            \node (2) at (-1, 0) {$\bullet$};
            \node (3) at (1,  0) {$\bullet$};
            %
            \draw[draw=lightgray] (1.center) -- (2.center);
            \draw[draw=lightgray] (1.center) -- (3.center);
        \end{tikzpicture}
        }
        & \raisebox{1cm}{$\cdots$}
        \raisebox{-0.1cm}{
        \begin{tikzpicture}[scale=0.6]
            \draw[help lines, color=gray!30, dashed] (-2.1,-1.1) grid (2.1, 1.5); 
            \node at (-1, 1.5) {\scalebox{0.5}{\scr{$N-j-1$}}\,\,\,};
            \node at (0, 1.5) {\scalebox{0.5}{\scr{$N-j$}}};
            \node at (1, 1.5) {\scalebox{0.5}{\,\,\,\scr{$N-j+1$}}};
            %
            \node (1) at (0, -1) {$\bullet$};
            \node (2) at (-1, 0) {$\bullet$};
            \node (3) at (1,  0) {$\bullet$};
            %
            \draw[draw=lightgray] (1.center) -- (2.center);
            \draw[draw=lightgray] (1.center) -- (3.center);
            \node at (0, -1.7) {\scr{\text{in the path $p_3$}}};
        \end{tikzpicture}
        }
\end{align*}
}
\!where the other parts of $p_1$, $p_2$ and $p_3$ are same.
By \eqref{eq:monomial map}, we have
\begin{equation*}
\begin{split}
    & \m(p_1) + \m(p_2) + \m(p_3) \\
    & \qquad =
    \left( 
        Y_{j,\ell-1}Y_{j,\ell+1}
        +
        Y_{j,\ell-1}Y_{j,\ell+1} A_{j,\ell+2}^{-1}
        +
        Y_{j,\ell-1}Y_{j,\ell+1} A_{j,\ell+2}^{-1} A_{j,\ell}^{-1}
    \right)M,
\end{split}
\end{equation*}
where $M \in \mathbb{Z}\left[Y_{\ov{l}, \ell}^{\pm 1}\right]_{(l, \ell) \in \X}$ has no any factor $Y_{j,p}^{\pm 1}$ for $p \in \Z$.
Since the $j$-th screening operator $S_j$ is a derivation,
it follows from \eqref{eq:description of ker Si} 
that $\m(p_1) + \m(p_2) + \m(p_3) \in \ker(S_j)$.
\smallskip

\noindent
{\it Case 4}. $|C|=4$.
Let $p_1$, $p_2$, $p_3$ and $p_4$ be the paths in $\ov{\Pa}_{i,k} \cap C$.
By the proof of Lemma \ref{lem:maximal length of j-component},
we assume that
\begin{equation*}
    p_2 = p_1 \A_{j,\ell}^{-1}, \quad
    p_3 = p_1 \A_{N-j, \ell'}, \quad
    p_4 = p_2 \A_{N-j, \ell'} = p_3 \A_{j, \ell},
\end{equation*}
where $\ell < \ell'$.
Then the local configurations of $p_1$, $p_2$, $p_3$ and $p_4$ are depicted as follows:
{\allowdisplaybreaks
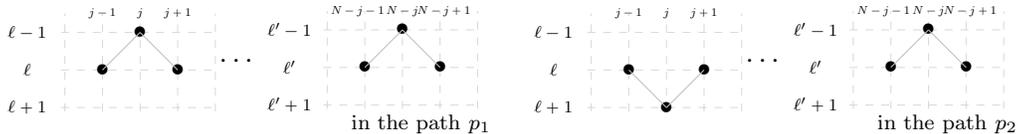
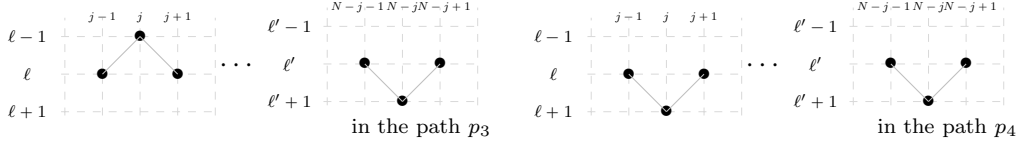
\begin{align*}
&
        \raisebox{0.25cm}{
        \begin{tikzpicture}[scale=0.5]
            \draw[help lines, color=gray!30, dashed] (-2.1,-1.1) grid (2.1, 1.5); 
            \node at (-1, 1.5) {\scalebox{0.55}{\scr{$j-1$}}};
            \node at (0, 1.5) {\scalebox{0.55}{\scr{$j$}}};
            \node at (1, 1.5) {\scalebox{0.55}{\scr{$j+1$}}};
            %
            \node at (-3, 1) {\scalebox{0.8}{\scr{$\ell-1$}}};
            \node at (-3, 0) {\scalebox{0.8}{\scr{$\ell$}}};
            \node at (-3, -1) {\scalebox{0.8}{\scr{$\ell+1$}}};
            \node (1) at (0, 1) {$\bullet$};
            \node (2) at (-1, 0) {$\bullet$};
            \node (3) at (1,  0) {$\bullet$};
            %
            \draw[draw=lightgray] (1.center) -- (2.center);
            \draw[draw=lightgray] (1.center) -- (3.center);
        \end{tikzpicture}}
        \raisebox{1cm}{$\cdots$}
        \begin{tikzpicture}[scale=0.5]
            \draw[help lines, color=gray!30, dashed] (-2.1,-1.1) grid (2.1, 1.5); 
            \node at (-1, 1.5) {\scalebox{0.55}{\scr{$N-j-1$}}\,\,\,};
            \node at (0, 1.5) {\scalebox{0.55}{\scr{$N-j$}}};
            \node at (1, 1.5) {\scalebox{0.55}{\,\,\,\scr{$N-j+1$}}};
            %
            \node at (-3, 1) {\scalebox{0.8}{\scr{$\ell'-1$}}};
            \node at (-3, 0) {\scalebox{0.8}{\scr{$\ell'$}}};
            \node at (-3, -1) {\scalebox{0.8}{\scr{$\ell'+1$}}};
            \node (1) at (0, 1) {$\bullet$};
            \node (2) at (-1, 0) {$\bullet$};
            \node (3) at (1,  0) {$\bullet$};
            %
            \draw[draw=lightgray] (1.center) -- (2.center);
            \draw[draw=lightgray] (1.center) -- (3.center);
            \node at (0.5, -1.5) {\scr{\text{in the path $p_1$}}};
        \end{tikzpicture}
        \,\,\,
        \raisebox{0.25cm}{
            \begin{tikzpicture}[scale=0.5]
                \draw[help lines, color=gray!30, dashed] (-2.1,-1.1) grid (2.1, 1.5); 
                \node at (-1, 1.5) {\scalebox{0.55}{\scr{$j-1$}}};
                \node at (0, 1.5) {\scalebox{0.55}{\scr{$j$}}};
                \node at (1, 1.5) {\scalebox{0.55}{\scr{$j+1$}}};
                %
                \node at (-3, 1) {\scalebox{0.8}{\scr{$\ell-1$}}};
                \node at (-3, 0) {\scalebox{0.8}{\scr{$\ell$}}};
                \node at (-3, -1) {\scalebox{0.8}{\scr{$\ell+1$}}};
                \node (1) at (0, -1) {$\bullet$};
                \node (2) at (-1, 0) {$\bullet$};
                \node (3) at (1,  0) {$\bullet$};
                %
                \draw[draw=lightgray] (1.center) -- (2.center);
                \draw[draw=lightgray] (1.center) -- (3.center);
            \end{tikzpicture}}
            \raisebox{1cm}{$\cdots$}
            \begin{tikzpicture}[scale=0.5]
                \draw[help lines, color=gray!30, dashed] (-2.1,-1.1) grid (2.1, 1.5); 
                \node at (-1, 1.5) {\scalebox{0.55}{\scr{$N-j-1$}}\,\,\,};
                \node at (0, 1.5) {\scalebox{0.55}{\scr{$N-j$}}};
                \node at (1, 1.5) {\scalebox{0.55}{\,\,\,\scr{$N-j+1$}}};
                %
                \node at (-3, 1) {\scalebox{0.8}{\scr{$\ell'-1$}}};
                \node at (-3, 0) {\scalebox{0.8}{\scr{$\ell'$}}};
                \node at (-3, -1) {\scalebox{0.8}{\scr{$\ell'+1$}}};
                \node (1) at (0, 1) {$\bullet$};
                \node (2) at (-1, 0) {$\bullet$};
                \node (3) at (1,  0) {$\bullet$};
                %
                \draw[draw=lightgray] (1.center) -- (2.center);
                \draw[draw=lightgray] (1.center) -- (3.center);
                \node at (0.5, -1.5) {\scr{\text{in the path $p_2$}}};
            \end{tikzpicture}
            \\
            &
            \raisebox{0.25cm}{
        \begin{tikzpicture}[scale=0.5]
            \draw[help lines, color=gray!30, dashed] (-2.1,-1.1) grid (2.1, 1.5); 
            \node at (-1, 1.5) {\scalebox{0.55}{\scr{$j-1$}}};
            \node at (0, 1.5) {\scalebox{0.55}{\scr{$j$}}};
            \node at (1, 1.5) {\scalebox{0.55}{\scr{$j+1$}}};
            %
            \node at (-3, 1) {\scalebox{0.8}{\scr{$\ell-1$}}};
            \node at (-3, 0) {\scalebox{0.8}{\scr{$\ell$}}};
            \node at (-3, -1) {\scalebox{0.8}{\scr{$\ell+1$}}};
            \node (1) at (0, 1) {$\bullet$};
            \node (2) at (-1, 0) {$\bullet$};
            \node (3) at (1,  0) {$\bullet$};
            %
            \draw[draw=lightgray] (1.center) -- (2.center);
            \draw[draw=lightgray] (1.center) -- (3.center);
        \end{tikzpicture}}
        \raisebox{1cm}{$\cdots$}
        \begin{tikzpicture}[scale=0.5]
            \draw[help lines, color=gray!30, dashed] (-2.1,-1.1) grid (2.1, 1.5); 
            \node at (-1, 1.5) {\scalebox{0.55}{\scr{$N-j-1$}}\,\,\,};
            \node at (0, 1.5) {\scalebox{0.55}{\scr{$N-j$}}};
            \node at (1, 1.5) {\scalebox{0.55}{\,\,\,\scr{$N-j+1$}}};
            %
            \node at (-3, 1) {\scalebox{0.8}{\scr{$\ell'-1$}}};
            \node at (-3, 0) {\scalebox{0.8}{\scr{$\ell'$}}};
            \node at (-3, -1) {\scalebox{0.8}{\scr{$\ell'+1$}}};
            \node (1) at (0, -1) {$\bullet$};
            \node (2) at (-1, 0) {$\bullet$};
            \node (3) at (1,  0) {$\bullet$};
            %
            \draw[draw=lightgray] (1.center) -- (2.center);
            \draw[draw=lightgray] (1.center) -- (3.center);
            \node at (0.5, -1.7) {\scr{\text{in the path $p_3$}}};
        \end{tikzpicture}
        \,\,\,
        \raisebox{0.25cm}{
        \begin{tikzpicture}[scale=0.5]
            \draw[help lines, color=gray!30, dashed] (-2.1,-1.1) grid (2.1, 1.5); 
            \node at (-1, 1.5) {\scalebox{0.55}{\scr{$j-1$}}};
            \node at (0, 1.5) {\scalebox{0.55}{\scr{$j$}}};
            \node at (1, 1.5) {\scalebox{0.55}{\scr{$j+1$}}};
            %
            \node at (-3, 1) {\scalebox{0.8}{\scr{$\ell-1$}}};
            \node at (-3, 0) {\scalebox{0.8}{\scr{$\ell$}}};
            \node at (-3, -1) {\scalebox{0.8}{\scr{$\ell+1$}}};
            \node (1) at (0, -1) {$\bullet$};
            \node (2) at (-1, 0) {$\bullet$};
            \node (3) at (1,  0) {$\bullet$};
            %
            \draw[draw=lightgray] (1.center) -- (2.center);
            \draw[draw=lightgray] (1.center) -- (3.center);
        \end{tikzpicture}}
        \raisebox{1cm}{$\cdots$}
        \begin{tikzpicture}[scale=0.5]
            \draw[help lines, color=gray!30, dashed] (-2.1,-1.1) grid (2.1, 1.5); 
            \node at (-1, 1.5) {\scalebox{0.55}{\scr{$N-j-1$}}\,\,\,};
            \node at (0, 1.5) {\scalebox{0.55}{\scr{$N-j$}}};
            \node at (1, 1.5) {\scalebox{0.55}{\,\,\,\scr{$N-j+1$}}};
            %
            \node at (-3, 1) {\scalebox{0.8}{\scr{$\ell'-1$}}};
            \node at (-3, 0) {\scalebox{0.8}{\scr{$\ell'$}}};
            \node at (-3, -1) {\scalebox{0.8}{\scr{$\ell'+1$}}};
            \node (1) at (0, -1) {$\bullet$};
            \node (2) at (-1, 0) {$\bullet$};
            \node (3) at (1,  0) {$\bullet$};
            %
            \draw[draw=lightgray] (1.center) -- (2.center);
            \draw[draw=lightgray] (1.center) -- (3.center);
            \node at (0.5, -1.7) {\scr{\text{in the path $p_4$}}};
        \end{tikzpicture}
\end{align*}
}
\!where the other parts of $p_1$, $p_2$, $p_3$ and $p_4$ are same.
By \eqref{eq:monomial map}, we have
\begin{equation*}
    \m(p_1) + \m(p_2) + \m(p_3) + \m(p_4) 
    =
    \left( 
        Y_{j,\ell-1} + Y_{j,\ell-1}A_{j,\ell}^{-1}
    \right)
    \left(
        Y_{j,\ell'-1} + Y_{j,\ell'-1}A_{j,\ell'}^{-1}
    \right)
    M
\end{equation*}
where $M \in \mathbb{Z}\left[Y_{\ov{l}, \ell}^{\pm 1}\right]_{(l, \ell) \in \X}$ has no any factor $Y_{j,p}^{\pm 1}$ for $p \in \Z$.
Since the $j$-th screening operator $S_j$ is a derivation,
it follows from \eqref{eq:description of ker Si} that $\m(p_1) + \m(p_2) + \m(p_3) + \m(p_4) \in \ker(S_j)$.
\smallskip

By {\it Case 1}--{\it Case 4}, we conclude that $\chi \in \ker(S_j)$ for any $j \in I$.
Hence, it follows from Theorem \ref{thm:Image of qchar is the intersection of kernels of Si} and Lemma \ref{lem:unqiue dominant monomial} that $\chi = \chi_q(L(Y_{i,k}))$. We complete the proof.
\end{proof}

As a byproduct of Lemma \ref{lem:thin property} and Theorem \ref{thm:main1},  we have the following property of $\chi_q(L(Y_{i,k}))$, which was already known in \cite{KS, H05}.

\begin{cor}
    Every coefficient of monomials in $\chi_q(L(Y_{i,k}))$ is $1$.
\end{cor}

\begin{rem}
{\em 
Let $V$ be a $U_q(\widehat{\g})$-module in $\mathcal{C}$. 
We say that $V$ is {\it special} if $\chi_q(V)$ has a unique dominant monomial.
It was known in \cite{FM01} that all fundamental modules are special.
In fact, since fundamental modules are special, Frenkel--Mukhin (FM) algorithm works correctly \cite[Theorem 5.9]{FM01}, that is, it generates a Laurent polynomial $\chi$ from $Y_{i,k}$ by expanding all possible $U_q(\widehat{sl}_2)$-strings while determining ``correct" coefficients, and then $\chi = \chi_q(L(Y_{i,k}))$ (see \cite[Section 5.5]{FM01} for the FM algorithm, cf.~\cite{NN11}).

On the other hand, we say that $V$ is {\it thin} if every $\ell$-weight space of $V$ has dimension $1$.
It was well-known (e.g.~\cite{CM06,H05,KS} for types $A_n^{(1)}$, $B_n^{(1)}$, and $C_n^{(1)}$, and \cite{H05} for type $G_2^{(1)}$) that every fundamental module is thin for types $A_n^{(1)}$, $B_n^{(1)}$, $C_n^{(1)}$, and $G_2^{(1)}$.
This is not true for other types. For example, for type $D_4^{(1)}$, $\chi_q(L(Y_{2,k}))$ has a monomial with coefficient $2$ (see \cite{Na03,Na10}).
}
\end{rem}

\begin{rem}
{\em 
In \cite{MY12}, the path description of $q$-characters for snake modules (including fundamental modules) in types $A_n^{(1)}$ and $B_n^{(1)}$ was proved by a criteria \cite[Theorem 3.4]{MY12} for {\it thin special} $q$-characters.
Indeed, the criteria does not depend on the choice of $\g$. 
Therefore, one can prove Theorem \ref{thm:main1} using \cite[Theorem 3.4]{MY12}.
However, the author hope to extend Theorem \ref{thm:main1} in the near future, so use more general argument based on Theorem \ref{thm:Image of qchar is the intersection of kernels of Si}.

We would like to remark that after this paper was submitted, Tong-Duan-Luo \cite{TDL} provided a path description for $q$-characters of fundamental modules in type $D_n^{(1)}$ following the proof of Theorem \ref{thm:main1}, where each path corresponds to monomial or bimonomial.
}
\end{rem}

\section{Examples and Further discussion} \label{sec:examples}

\subsection{Examples} \label{subsec:examples}
Let us illustrate the path description of $\chi_q(L(Y_{i,k}))$ for $i \in I$ and $k \in \Z_+$ in type $C_3$ up to shift of spectral parameters (i.e.~vertical parameters under the path description below).
In the following examples, an upper (resp.~lower) corner is marked as a red (resp.~blue) bullet. For a path $p \in \ov{\Pa}_{i,k}$, we indicate the corresponding monomial $\m(p)$ in the bottom right of $p$.
Let us recall \eqref{eq:monomial map}.

\begin{rem} \label{rem:computing q-characters}
{\em 
One may obtain an explicit formula of $\chi_q(L(Y_{i,k}))$ from FM algorithm \cite{FM01} (cf.~\cite{NN11}) or Hernandez--Leclerc (HL) algorithm \cite{HL16} using cluster algebras structure \cite{HL10}.
In particular, the latter one is available by a computer program\footnote{\scalebox{0.75}{ \url{https://sites.google.com/view/isjang/side-project/hernandez-leclerc-algorithm-for-q-characters}}} based on the cluster algebra package \cite{MS10} in SAGEMATH \cite{SAGE}.
For convenience, the explicit formulae of $\chi_q(L(Y_{i,k}))$ for low ranks can be found in the private note by the author of this paper\footnote{\,\scalebox{0.75}{\url{https://github.com/ILSEUNGJANG/Algorithms/blob/main/HLalgorithm/Examples/qchar_graphs.pdf}}}, which are computed by FM algorithm (also checked by HL algorithm).
}
\end{rem}

\subsubsection{First fundamental modules}
In the case of $i = 1$, the number of monomials in $\chi_q(L(Y_{1,k}))$ is $6$. 
Then the path description of $\chi_q(L(Y_{1,0}))$ is given as follows:
{\allowdisplaybreaks
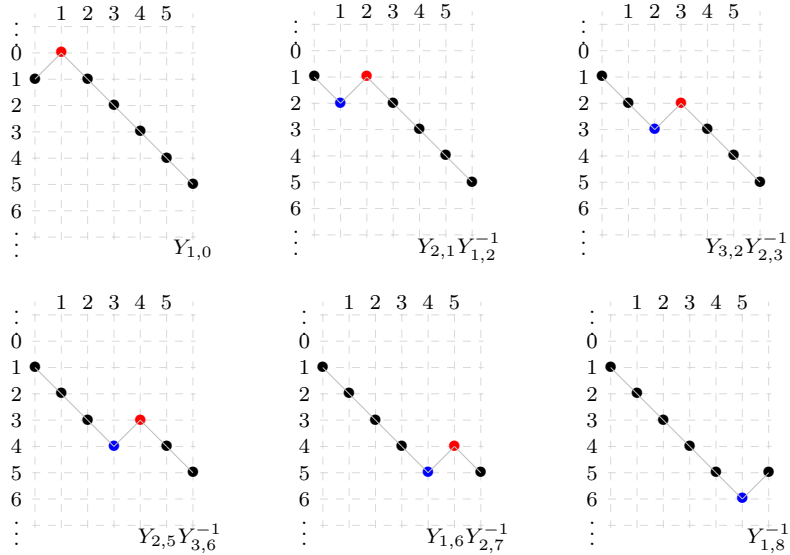
\begin{align*}
&

\end{align*}
}

\subsubsection{Second fundamental modules}
In the case of $i = 2$, the number of monomials in $\chi_q(L(Y_{2,k}))$ is $14$. 
Then the path description of $\chi_q(L(Y_{2,1}))$ is given as follows:
\vspace{-0.1cm}
{\allowdisplaybreaks
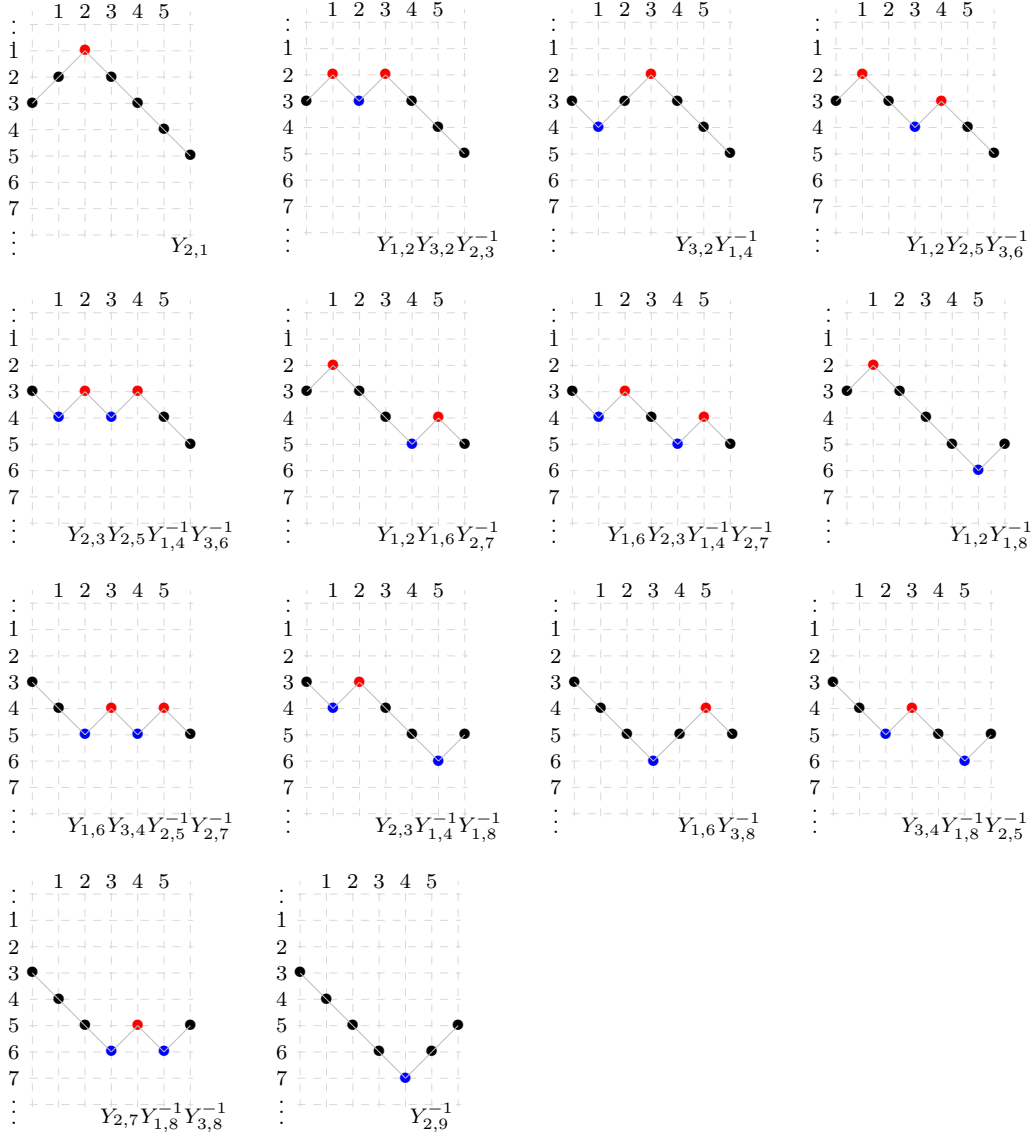
\begin{align*}
&
%
\end{align*}
}

\subsubsection{Third fundamental modules}
In the case of $i = 3$, the number of monomials in $\chi_q(L(Y_{3,k}))$ is $14$. 
Then the path description of $\chi_q(L(Y_{3,0}))$ is given as follows:
{\allowdisplaybreaks
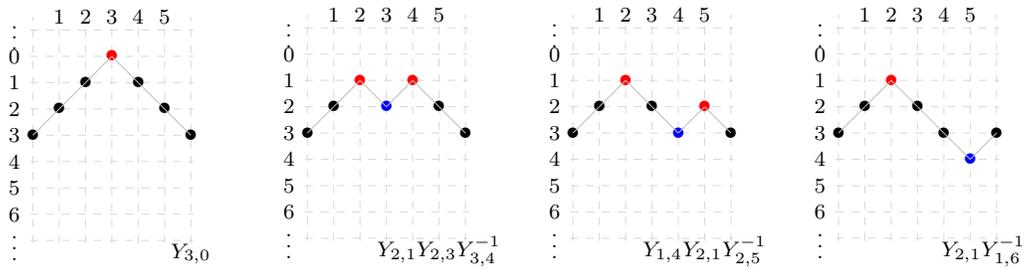
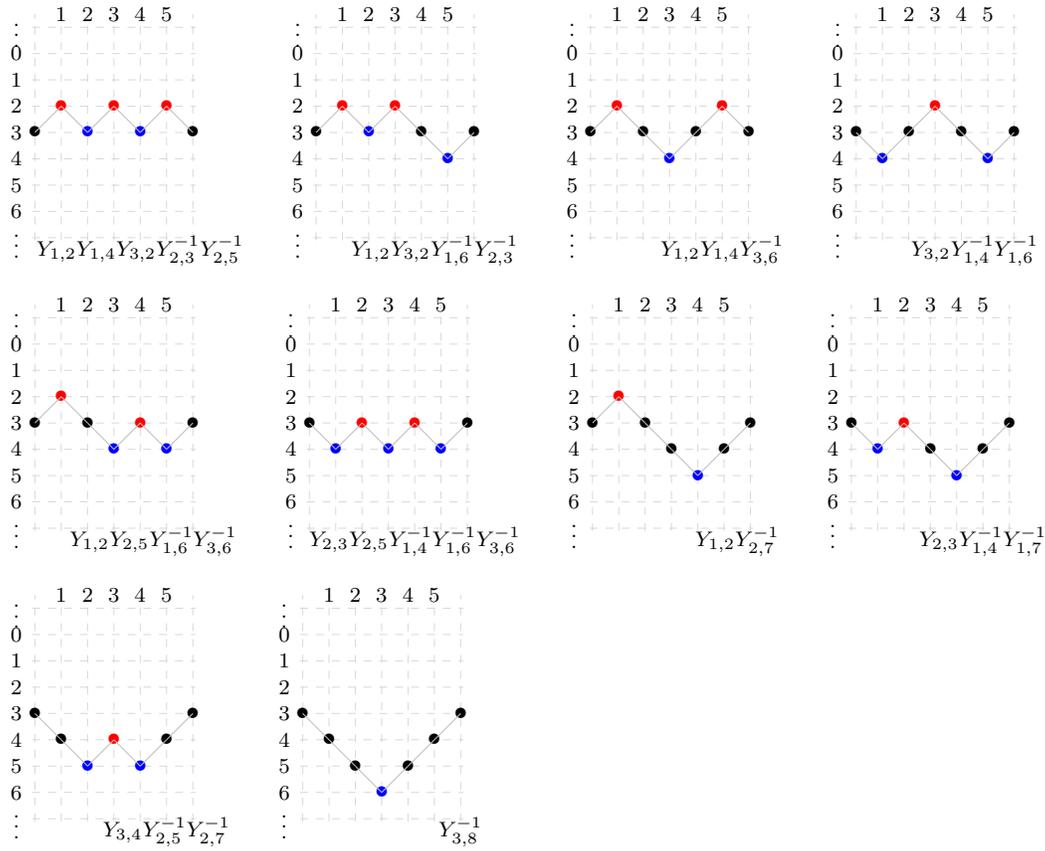
\begin{align*}
&
%
\end{align*}
}
\smallskip

\subsection{Further discussion} \label{subsec:remarks}
It is natural to ask whether there is a notion of non-overlapping paths for type C as in \cite{MY12}. 
But, there is an example in which a monomial corresponding to an overlapping path is contained in the $q$-character of a simple module under the path description. 
Let us give an example for type $C_3$.
We consider the $q$-character $\chi_q(Y_{2,1}Y_{2,3})$ and the following overlapping paths as shown below:
{\allowdisplaybreaks
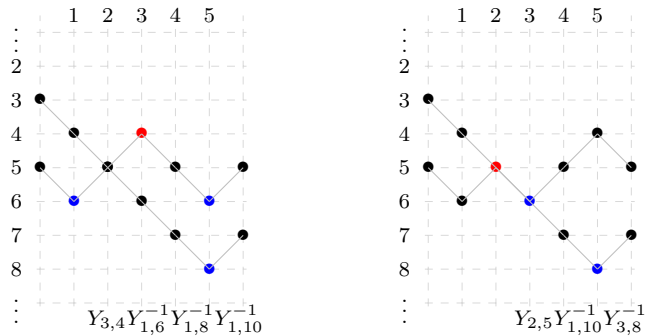
\begin{align*}
        %
\end{align*}
    where the non-canceled upper (resp.~lower) corner is marked as the red (resp.~blue) bullet. One may check that the corresponding monomials are contained in $\chi_q(Y_{2,1}) \cdot \chi_q(Y_{2,3})$.
    However, we have
    \begin{equation*}
        Y_{3,4} Y_{1,6}^{-1} Y_{1,8}^{-1} Y_{1,10}^{-1} \notin \mcM(Y_{2,1}Y_{2,3}),\quad
        Y_{2,5} Y_{1,10}^{-1} Y_{3,8}^{-1} \in \mcM(Y_{2,1}Y_{2,3})
    \end{equation*}
    by considering the $T$-system \cite{KNS} that is a relation on the $q$-characters of Kirillov--Reshetikhin modules (proved in \cite{Na03b} for simply laced types and \cite{H06, H10} for general types) and computing $\chi_q(Y_{2,1}Y_{2,3})$ (by the methods explained in Remark \ref{rem:computing q-characters}).
    
    Hence, the notion of (non-)overlapping paths does not seem to be enough to distinguish monomials for a simple quotient from an ordered tensor product of fundamental modules, even in the case of Kirillov--Reshetikhin modules.

    Recently, it was shown in \cite{GDL} that there is a path description for type A Hernandez--Leclerc modules \cite{HL16} in terms of a notion of compatible condition on paths, which is a generalization of the non-overlapping paths in \cite{MY12}.
    Note that the path description in \cite{GDL} allows overlapping paths.
    Since the underlying combinatorics of paths in type C is almost identical to the one of type A in \cite{MY12}, it would be interesting to find a suitable compatible condition on paths as in \cite{GDL} to extend the formula in Theorem \ref{thm:main1}.
    

\providecommand{\bysame}{\leavevmode\hbox to3em{\hrulefill}\thinspace}
\providecommand{\MR}{\relax\ifhmode\unskip\space\fi MR }
\providecommand{\MRhref}[2]{%
  \href{http://www.ams.org/mathscinet-getitem?mr=#1}{#2}
}
\providecommand{\href}[2]{#2}

\end{document}